\documentclass[reqno,12pt,oneside]{amsart}

\usepackage{amsthm,amsmath,amssymb,amsfonts}
\usepackage[colorlinks=true,hypertexnames=false,linkcolor=blue,citecolor=blue]{hyperref}
\usepackage[top=3cm,bottom=2.5cm,left=2.5cm, right=3cm,marginparwidth=0cm]{geometry}
\usepackage{epsfig,graphics}
\usepackage{amscd}
\usepackage{amsmath}
\usepackage{mathtools}
\usepackage{bm}
\usepackage{amssymb}
\usepackage{graphicx}
\usepackage{psfrag}
\usepackage{verbatim}
\usepackage{tikz}
\usetikzlibrary{quotes}
\usetikzlibrary{topaths}
\usetikzlibrary{shapes,arrows,cd}
\usepackage{adjustbox}

\newtheorem{definition}{Definition}[section]
\newtheorem{theorem}{Theorem}[section]
\newtheorem{lemma}[theorem]{Lemma}
\newtheorem{corollary}[theorem]{Corollary}
\newtheorem{proposition}{Proposition}[section]

\theoremstyle{remark}
\newtheorem{remark}[theorem]{Remark}

\numberwithin{equation}{section}

\theoremstyle{plain}
\newtheorem{maintheorem}{Theorem}

\newcommand{\R}{\ensuremath{\mathbb{R}}}
\newcommand{\Z}{\ensuremath{\mathbb{Z}}}
\newcommand{\C}{\ensuremath{\mathbb{C}}}
\newcommand{\nt}{\ensuremath{\mathbb{N}}}

\newcommand{\crit}{\operatorname{Crit}}
\newcommand{\Id}{\operatorname{Id}}

\begin{document}

\title[Renormalization of bicritical circle maps]{Renormalization of bicritical circle maps}

\author{Gabriela Estevez}
\address{Instituto de Matemática, Universidade Federal do Rio de Janeiro}
\curraddr{Av. Athos da Silveira Ramos 149, CEP 21945-909. Rio de Janeiro, RJ, Brazil}
\email{gaestevezja@gmail.com}

\author{Pablo Guarino}
\address{Instituto de Matem\'atica e Estat\'istica, Universidade Federal Fluminense}
\curraddr{Rua Prof. Marcos Waldemar de Freitas Reis, S/N, 24.210-201, Bloco H, Campus do Gragoat\'a, Niter\'oi, Rio de Janeiro RJ, Brasil}
\email{pablo\_\,guarino@id.uff.br}



\thanks{G.E. was partially financed by the Coordenação de Aperfeiçoamento de Pessoal de Nível Superior - Brasil (CAPES) - Finance Code 001. P.G. was partially supported by Conselho Nacional de Desenvolvimento Cient\'ifico e Tecnol\'ogico (CNPq) and by Coordena\c{c}\~ao de Aperfei\c{c}oamento de Pessoal de N\'ivel Superior - Brasil (CAPES) grant 23038.009189/2013-05.}

\begin{abstract} A general \emph{ansatz} in Renormalization Theory, already established in many important situations, states that exponential convergence of renormalization orbits implies that topological conjugacies are actually smooth (when restricted to the attractors of the original systems). In this paper we establish this principle for a large class of \emph{bicritical circle maps}, which are $C^3$ circle homeomorphisms with irrational rotation number and exactly two (non-flat) critical points. The proof presented here is an adaptation, to the bicritical setting, of the one given by de Faria and de Melo in \cite{dFdM} for the case of a single critical point. When combined with the recent papers \cite{ESY,Yam2019}, our main theorem implies $C^{1+\alpha}$ rigidity for real-analytic bicritical circle maps with rotation number of \emph{bounded type} (Corollary \ref{maincoro}).
\end{abstract}

\maketitle

\section{Introduction}

In the present paper we study the dynamics of \emph{multicritical circle maps}, which are $C^3$ circle homeomorphisms having finitely many critical points (all of which are non-flat, see Definition \ref{def:multicritic}). By a fundamental result due to J.-C. Yoccoz \cite{yoccoz}, two multicritical circle maps $f$ and $g$ with the same irrational rotation number are topologically conjugate to the corresponding rigid rotation, and in particular to each other. To obtain a smooth conjugacy between $f$ and $g$, we need to assume the existence of a topological conjugacy $h$ which identifies its critical sets, while preserving corresponding criticalities (this is a \emph{finite codimension} condition, see Definition \ref{signature}).

In this paper we restrict our attention to the bicritical case, and we prove that, for Lebesgue almost every rotation number, such conjugacy $h$ is a $C^{1+\alpha}$ diffeomorphism, provided the successive renormalizations of $f$ and $g$ (around critical points identified under $h$) converge together exponentially fast in the $C^1$ topology (Theorem \ref{maintheorem}). The full Lebesgue measure set of rotation numbers considered here was introduced by de Faria and de Melo in the nineties \cite{dFdM}, and contains all numbers of bounded type (see Definition \ref{setA}). As already mentioned in the abstract, the proof presented here is an adaptation, to the bicritical setting, of the one given in \cite{dFdM} for the case of a single critical point.

As an application, we combine Theorem \ref{maintheorem} with the recent papers \cite{ESY,Yam2019} to obtain $C^{1+\alpha}$ rigidity for real-analytic bicritical circle maps with bounded combinatorics (Corollary~\ref{maincoro}).

\subsection{Main result} Let $f$ be a $C^3$ multicritical circle map with irrational rotation number $\rho\in(0,1)$ and $N \geq 1$ critical points $c_i$, for $0\leq i\leq N-1$ (which are labeled as ordered in the unit circle). All critical points are assumed to be \emph{non-flat}: in $C^3$ local coordinates around $c_i$, the map $f$ can be written as\, $t \mapsto t\,|t|^{d_i-1}$ for some $d_i>1$ (we say that $d_i$ is the \emph{criticality} of $f$ around $c_i$, see Definition \ref{naoflat} below). Being topologically conjugate to an irrational rotation, $f$ is uniquely ergodic; we denote its unique invariant Borel probability measure by $\mu_f$.

\begin{definition}\label{signature} We define the \emph{signature} of $f$ to be the $(2N+2)$-tuple 
\[
(\rho\,;N;\,d_0,d_1,\ldots,d_{N-1};\,\delta_0,\delta_1,\ldots,\delta_{N-1}),
\]
where  $d_i$ is the criticality of the critical point $c_i$, and $\delta_i=\mu_f[c_i,c_{i+1})$ (with the convention that $c_{N}=c_0$).
\end{definition}

Now consider two $C^3$ multicritical circle maps, say $f$ and $g$, with the same irrational rotation number. By Yoccoz theorem \cite{yoccoz}, they are topologically conjugate to each other. By elementary reasons, if $f$ and $g$ have the same signature there exists a circle homeomorphism $h$, which is a topological conjugacy between $f$ and $g$, identifying each critical point of $f$ with one of $g$ having the same criticality (note that such $h$ is the unique conjugacy between $f$ and $g$ that can be smooth. As it turns out, for almost every rotation number most conjugacies between $f$ and $g$ fail to be a \emph{quasisymmetric} homeomorphism, see the recent paper \cite{dFG19} for precise statements). This will be our standing assumption in this article. In particular, a critical point of $f$ and one of $g$ are said to be \emph{corresponding} critical points, if they are identified under such conjugacy $h$. Our main result is the following.

\begin{maintheorem}\label{maintheorem} There exists a full Lebesgue measure set $\mathcal{A}\subset(0,1)$ of irrational numbers with the following property. Let $f$ and $g$ be $C^3$ bicritical circle maps with the same signature and such that its common rotation number belongs to the set $\mathcal{A}$. If the renormalizations of $f$ and $g$ around corresponding critical points converge together exponentially fast in the $C^1$ topology, then $f$ and $g$ are conjugate to each other by a $C^{1+\alpha}$ diffeomorphism, for some $\alpha>0$.
\end{maintheorem}

As mentioned in the abstract, the idea that exponential convergence of renormalization implies smoothness of topological conjugacies, when restricted to the attractors of the original systems, is a cornerstone in Renormalization Theory. As a fundamental example, see \cite[Section VI.9]{dMvS} for the case of unimodal maps with bounded combinatorics (more specifically, see Theorem 9.4). In the case of critical circle maps with a single critical point, this principle has been established by de Faria and de Melo in \cite[First Main Theorem]{dFdM} for rotation numbers in the set $\mathcal{A}$, and extended by Khanin and Teplinsky in \cite{KT} to cover  all irrational rotation numbers (see Theorem 2 in page 198 for the specific statement). Both proofs are given for the case of a single critical point, and our goal in the present paper is to adapt the previous arguments to the case of two critical points.

We would like to remark that Theorem \ref{maintheorem} is most likely true for circle maps with \emph{any} number of critical points, see Remark \ref{renbimulti} at the end of the present paper. On the other hand, it is definitely not possible to extend its statement to cover \emph{all} irrational rotation numbers: in \cite{Av}, Avila was able to construct topologically conjugate real-analytic critical circle maps (with a single critical point) which are \emph{not} $C^{1+\alpha}$ conjugate to each other, for any $\alpha>0$, although the corresponding renormalization orbits converge together exponentially fast (in the $C^r$ metric, for any $r \geq 1$). We remark that an analogue statement, in the $C^{\infty}$ class, was previously obtained in \cite[Section 5]{dFdM}. However, $C^1$ rigidity may hold for multicritical circle maps with the same signature, just as in the case of a single critical point. Indeed, any two $C^3$ circle homeomorphisms with the same irrational rotation number of bounded type and with a single critical point (of the same odd integer criticality) are conjugate to each other by a $C^{1+\alpha}$ circle diffeomorphism, for some universal $\alpha>0$ (see \cite{GdM}). Moreover, any two $C^4$ circle homeomorphisms with the same irrational rotation number and with a unique critical point (again, of the same odd criticality), are conjugate to each other by a $C^1$ diffeomorphism (see \cite{GMdM}). This conjugacy is in fact a $C^{1+\alpha}$ diffeomorphism, provided the common rotation number belongs to the full Lebesgue measure set $\mathcal{A}$ (again, see \cite{GMdM}).

\subsection{Rigidity of real-analytic bicritical circle maps with bounded combinatorics}\label{seccoroBT} Let $f$ and $g$ be real-analytic bicritical circle maps with both critical points of cubic type and with the same signature (recall, from Definition \ref{signature}, that this amounts to say that $f$ and $g$ have the same irrational rotation number while the relative positions of its two critical points, viewed with the corresponding unique invariant measure, coincide). If the common rotation number of $f$ and $g$ is of \emph{bounded type}, it follows from the recent papers \cite{ESY,Yam2019} that the successive renormalizations of $f$ and $g$, around corresponding critical points, converge together exponentially fast in the $C^r$ topology, for any $r\in\nt$. Applying Theorem \ref{maintheorem}, we obtain the following result, announced in the abstract and in the introduction.

\begin{corollary}\label{maincoro} Let $f$ and $g$ be real-analytic bicritical circle maps with the same signature, and with both critical points of cubic type. If their common rotation number is of bounded type, then $f$ and $g$ are conjugate to each other by a $C^{1+\alpha}$ diffeomorphism for some $\alpha>0$.
\end{corollary}

\subsection{Strategy of the proof of Theorem \ref{maintheorem}}\label{secsketch} Let $f$ and $g$ be two $C^3$ bicritical circle maps with the same irrational rotation number. As explained in the introduction, a result of Yoccoz \cite{yoccoz} implies that $f$ and $g$ are topologically conjugate to each other. Moreover, assuming that $f$ and $g$ have the same signature is equivalent to assume that there exists a circle homeomorphism $h$ which is a topological conjugacy between $f$ and $g$, identifying each critical point of $f$ with one of $g$ having the same criticality.

Our main goal in this paper is to prove that such homeomorphism $h$ is actually a smooth diffeomorphism. Since one-dimensional affine maps are characterized by the fact that they preserve ratios between lengths of intervals, we would like to show that $h$ almost preserves such ratios, provided we consider very small intervals, which are very close to each other. To achieve this, we will first construct (say, for the given $f$) a suitable sequence of partitions (called \emph{fine grid}, see Definition \ref{def:finegrid}) whose vertices will be dynamically extracted from the critical set of $f$. Essentially, this is a combinatorial construction, to be performed in Section~\ref{sec:anewpartition}.

After fine grids are built, it will be enough to control ratios of lengths of corresponding elements of those fine grids, in order to assure that $h$ is indeed a $C^{1+\alpha}$ diffeomorphism (see Proposition \ref{criterion}). This criterion was used by de Faria and de Melo in \cite{dFdM} to obtain smoothness (see \cite[Section 4.2]{dFdM}), and it is the one that will be used here too. Let us point out that the fine grids constructed in \cite{dFdM} are not suitable for the case of more than one critical point. As already mentioned, in Section \ref{sec:anewpartition} we will construct fine grids adapted to the bicritical case, which is the main difference between the proof given here and the one given by de Faria and de Melo for the case of a single critical point.

Our main task in this paper, therefore, is to prove that, for Lebesgue almost every rotation number, $C^1$ exponential contraction of renormalization (which is the main assumption of Theorem \ref{maintheorem}) implies the local behaviour for $h$ required by Proposition \ref{criterion}. This is accomplished in Section \ref{SecProofThmB}.

We finish this introduction by pointing out that to prove exponential contraction for the renormalization operator of multicritical circle maps is a challenging problem. In the case of a single critical point and real-analytic dynamics, exponential contraction was obtained in \cite{dFdM2} for rotation numbers of bounded type, and extended in \cite{KY} to cover all irrational rotation numbers. Both papers lean on complex dynamics techniques, and therefore an additional hypothesis is required: the criticality at both critical points has to be an odd integer (note that this condition is also needed for the rigidity results discussed after the statement of Theorem \ref{maintheorem}). These results have been recently extended in at least two directions: in \cite{GY} exponential contraction is obtained allowing non-integer criticalities which are close enough to an odd integer, while in \cite{GdM} and \cite{GMdM} exponential contraction is established for finitely smooth critical circle maps (still with odd integer criticalities). Finally, in the case of two critical points, it was recently proved in \cite{Yam2019} both existence of periodic orbits and hyperbolicity of those periodic orbits, for real-analytic bicritical circle maps (with both critical points of cubic type). These results were later extended to bounded combinatorics in \cite{ESY}, from where we deduce Corollary \ref{maincoro} (as already explained in Section \ref{seccoroBT} above). For much more on the dynamics of multicritical circle maps, we refer the reader to the recent survey \cite{dFGsurvey}.

\subsection*{Brief summary} In the preliminary Section \ref{preliminares} we present the basic facts about multicritical circle maps and renormalization of commuting pairs. In Section \ref{secreduction} we state Theorem \ref{theoremb}, and we explain why it implies Theorem \ref{maintheorem}. The two remaining sections are devoted to the proof of Theorem \ref{theoremb}\,: in Section \ref{sec:anewpartition} we state the announced criterion for smoothness (Proposition \ref{criterion}) and we construct a fine grid for any given bicritical circle map, while in Section \ref{SecProofThmB} we prove Theorem \ref{theoremb} by establishing the assumptions of Proposition \ref{criterion}.

\section{Preliminaries}\label{preliminares}

\subsection{Bicritical circle maps}\label{subseccomb} Let us now define the maps which are the main object of study in the present paper. We start with the notion of {\it non-flat critical point\/}. 

\begin{definition}\label{naoflat} We say that a critical point $c$ of a one-dimensional $C^3$ map $f$ is \emph{non-flat} of criticality $d>1$ if there exists a neighbourhood $W$ of the critical point such that\, $f(x)=f(c)+\phi(x)\,\big|\phi(x)\big|^{d-1}$ for all $x \in W$, where $\phi : W \rightarrow \phi(W)$ is 
an orientation preserving $C^{3}$ diffeomorphism satisfying $\phi(c)=0$.
\end{definition}

\begin{definition}\label{def:multicritic} A \emph{multicritical circle map} is an orientation preserving $C^3$ circle homeomorphism $f$ having 
$N \geq 1$ critical points, all of which are non-flat in the sense of Definition~\ref{naoflat}. If $N=2$, we say that $f$ is a \emph{bicritical circle map}.
\end{definition}

As an example, let $a \in [0,1)$, $N\in\nt$ and consider $\tilde{f}_a:\R\to\R$ given by$$\tilde{f}_a(x)=x+a-\frac{1}{2N \pi}\,\sin (2 N \pi x).$$Since each $\tilde{f}_a$ has degree one and commutes with unitary translation, it is the lift of an orientation preserving real-analytic circle homeomorphism, under the canonical universal cover $x \mapsto e^{2\pi ix}$. Each circle map in this family has exactly $N$ critical points, given by $\big\{e^{\frac{j}{N}2\pi i}:\,j\in\{0,1,...,N-1\}\big\}$, all of them with criticality equal to $3$. Since they lift to entire maps, these real-analytic multicritical circle maps extend holomorphically to the punctured plane $\C\setminus\{0\}$. It is also possible to construct multicritical circle maps whose holomorphic extensions are well defined in the whole Riemann sphere $\widehat{\C}$. Indeed, consider the one-parameter family $f_{\omega}:\widehat{\C}\to\widehat{\C}$ of 
Blaschke products given by$$
f_{\omega}(z)=e^{2\pi i\omega}z^2\left(\frac{z-3}{1-3z}\right)\quad\mbox{for $\omega\in[0,1)$.}$$Every map in this family leaves invariant the unit circle, and restricts to a real-analytic critical circle map with a single critical point at $1$, which is of cubic type, and with critical value $e^{2\pi i\omega}$. Moreover, by monotonicity of the rotation number, for each $\rho\in(0,1)\!\setminus\!\mathbb{Q}$ there exists a unique $\omega$ in $[0,1)$ such that the rotation number of $f_{\omega}|_{S^1}$ equals $\rho$ (see \cite[Section 6]{dFG16} for more details). Now let $p,q \in \mathbb{C}$ with $|p|>1$, $|q|>1$, let $\omega\in[0,1)$ and consider $g_{p,q,\omega}:\widehat{\C}\to\widehat{\C}$ given by
\begin{equation}\label{Zakeri}
g_{p,q,\omega}(z)=e^{2\pi i\omega}z^{3}\left(\frac{z-p}{1-\overline{p}z}\right)\left(\frac{z-q}{1-\overline{q}z} \right).
\end{equation}Just as before, every map in this family leaves invariant the unit circle. The following fact was proved by Zakeri in \cite[Section 7]{zak}.

\begin{theorem} For any given $\rho\in(0,1)\!\setminus\!\mathbb{Q}$ and $\delta\in(0,1)$ there exists a unique $g_{p,q,\omega}$ of the form \eqref{Zakeri} such that $g_{p,q,\omega}|_{S^{1}}$ is a bicritical circle map with signature $(\rho\,;2;3,3;\,\delta,1-\delta)$.
\end{theorem}

\subsection{Real bounds}\label{secrealbounds} Being a homeomorphism, a multicritical circle map $f$ has a well defined rotation number. We will focus on the case where $f$ has no periodic orbits, which is equivalent to say that it has irrational rotation number $\rho \in [0,1]$. By the already mentioned result of J.-C. Yoccoz \cite{yoccoz}, $f$ has no wandering intervals and in particular it is topologically conjugate with the corresponding rigid rotation.

We consider the continued fraction expansion of $\rho$\,:
\begin{equation*}
\rho= [a_{0} , a_{1} , \cdots ]=\cfrac{1}{a_{0}+\cfrac{1}{a_{1}+\cfrac{1}{ \ddots} }} \ .
\end{equation*}Truncating the expansion at level $n-1$, we obtain the so-called \textit{convergents} of $\rho$\,:$$\frac{p_n}{q_n}\;=\;[a_0,a_1, \cdots ,a_{n-1}]\;=\;\dfrac{1}{a_0+\dfrac{1}{a_1+\dfrac{1}{\ddots\dfrac{1}{a_{n-1}}}}}\ .
$$The sequence of denominators $\{q_n\}_{n\in\nt}$ satisfies the following recursive formula (see for instance \cite[Chapter I, Theorem 1, page 4]{khin}):
\begin{equation*}
q_{0}=1, \hspace{0.4cm} q_{1}=a_{0}, \hspace{0.4cm} q_{n+1}=a_{n}\,q_{n}+q_{n-1} \hspace{0.3cm} \text{for all $n \geq 1$} .
\end{equation*}

As mentioned in the introduction, the set $\mathcal{A}\subset(0,1)$ considered in the statement of Theorem \ref{maintheorem} was introduced by de Faria and de Melo in the nineties \cite[Section 4.4]{dFdM}. Its precise definition is the following.

\begin{definition}\label{setA} Let $\mathcal{A}\subset(0,1)$ be the set of rotation numbers $\rho=[a_0, a_1, \cdots]$ satisfying:
\begin{enumerate}
\item\label{cond1A} $\displaystyle\limsup_{n\to \infty} \frac{1}{n} \sum_{j=0}^{j=n} \log a_j < \infty$\,,
\vspace{.2cm}
\item\label{cond2A} $\displaystyle\lim_{n \to \infty} \frac{1}{n}\, \log a_n=0$\,,
\vspace{.2cm}
\item\label{cond3A} $\displaystyle\frac{1}{n} \sum_{j=k+1}^{k+n} \log a_j \leq \omega\left(\frac{n}{k}\right)$\,,
\end{enumerate}
for all $0 < n \leq k$, where $\omega$ is a monotone function (that depends on the rotation number) such that $\omega (t)>0$ for all $t>0$, and such that $t\,\omega(t) \to 0$ as $t \to 0$.
\end{definition}

The set $\mathcal{A}$ has full Lebesgue measure in $(0,1)$, see \cite[Appendix C]{dFdM} for a proof. Obviously, all bounded type numbers satisfy the three conditions above (recall that $\rho$ is of \textit{bounded type} if $\sup_{n\in\nt}\{a_n\}$ is finite). We would like to remark that all constructions to be performed in Section \ref{sec:anewpartition} can be done for any irrational rotation number: conditions \eqref{cond1A}-\eqref{cond3A} in Definition \ref{setA} will only be considered in Section \ref{SecProofThmB}.

\medskip

Let $f$ a multicritical circle map, $x \in S^1$ and $n \in \nt$. We denote by $I_{n}(x)$ the interval
with endpoints $x$ and $f^{q_n}(x)$, which contains the point $f^{q_{n+2}}(x)$. The collection of intervals
\[
 \mathcal{P}_n(x) \ = \ \left\{ f^{i}(I_n(x)):\;0\leq i\leq q_{n+1}-1 \right\} \;\cup\; 
 \left\{ f^{j}(I_{n+1}(x)):\;0\leq j\leq q_{n}-1 \right\} 
\]is a partition of the circle (modulo endpoints) called the {\it standard $n$-th dynamical partition\/} associated to the point $x$. The following fundamental geometric control was obtained by Herman \cite{H} and Światek \cite{G} in the eighties.
  
\begin{theorem}[The real bounds]\label{teobeau} Given $N \in \nt$ and $d>1$ let $\mathcal{F}_{N,d}$ be the family of multicritical circle maps with at most $N$ critical points whose maximum criticality is bounded by $d$. There exists a constant $C=C(N,d)>1$ with the following property: for any given $f \in \mathcal{F}_{N,d}$ and $c \in \crit(f)$ there exists $n_0 \in \mathbb{N}$ such that for all $n\geq n_0$ and every pair of adjacent intervals $I,J \in \mathcal{P}_{n}(c)$, we have$$ \frac{1}{C} \leq \frac{|I|}{|J|} \leq C\,.$$
\end{theorem}

A detailed proof of Theorem \ref{teobeau} can also be found in \cite{EdF,EdFG}. Given two positive numbers $\alpha$ and $\beta$, we say that $\alpha$ is {\it comparable\/} to $\beta$ if there exists a constant $K>1$ depending only on $C$ (from Theorem \ref{teobeau}) such that $K^{-1}\beta\leq \alpha\leq K\beta$. This relation will be denoted $\alpha\asymp \beta$. We finish Section \ref{secrealbounds} with the following four consequences of the real bounds, that will be useful later.

\begin{remark}\label{remcomp} Let $I\in\mathcal{P}_n(c)$ and let $J$ be an interval such that $I \subseteq J \subseteq I^*$, where $I^*$ denotes the union of $I$ with its left and right neighbours in $\mathcal{P}_n(c)$. Then $|I|\asymp|J|$.
\end{remark}

\begin{corollary}\label{corbeau} Let $f$ be a multicritical circle map and $c \in \crit(f)$. There exists a constant\, $0<\mu<1$ such that for all $n\geq n_0$ the following holds: if $\mathcal{P}_{n+1}(c) \ni I \subsetneq J \in \mathcal{P}_{n}(c)$, then $|I| \leq \mu |J|$.
\end{corollary}

\begin{corollary}[$C^1-$bounds]\label{c1bounds} Let $f$ be a multicritical circle map and $c \in \crit(f)$. There exists a constant $K=K(f)>1$ such that for all $n >n_0$ and $x \in J_n(c)$, we have
\begin{itemize}
\item $Df^{q_{n+1}}(x) \leq K$\,,\, if $x \in I_n(c)$.
\item $Df^{q_{n}}(x) \leq K$\,,\, if $x \in I_{n+1}(c)$.
\end{itemize}
\end{corollary}
 
We say that two adjacent intervals $I$ and $J$ are \textit{symmetric} if their extreme points are $f^{-q_n}(x),x,f^{q_n}(x)$ for some $x \in S^1$ and $n \in \nt$.
 
\begin{corollary}\label{symmetricintervals} Any two adjacent symmetric intervals are comparable to each other.
\end{corollary}

Both Remark \ref{remcomp} and Corollary \ref{corbeau} follow straightforward from the real bounds. For a proof of Corollary \ref{c1bounds} see \cite[Lemma 3.1]{EdFG}, and for a proof of Corollary \ref{symmetricintervals} see \cite[Lemma 3.3]{EdF}.

\subsection{Multicritical commuting pairs}\label{secdefmulti} In this section, we introduce the notion of \textit{multicritical commuting pairs}, a natural generalization of the already classical notion of critical commuting pairs.

\begin{definition}\label{multcommpairs} A $C^r$ multicritical commuting pair ($r \geq 3$) with $N=N_1+N_2-1$ critical points is a pair $\zeta=(\eta,\xi)$ consisting of two $C^r$ orientation preserving homeomorphisms $\xi:I_\xi \rightarrow  \xi(I_\xi)$ and $\eta : I_\eta \rightarrow \eta(I_{\eta})$ with a finite number of non-flat critical points $\gamma_0, \cdots, \gamma_{N_1-1}=0$ and $\beta_{0}=0, \beta_1, \cdots, \beta_{N_2-1}$, respectively, satisfying:
 \begin{enumerate}
  \item $I_{\xi}=[\eta(0), 0]$ and $I_{\eta}=[0, \xi(0)]$ are compact intervals in the real line;
  \item $(\eta \circ \xi)(0)=(\xi \circ \eta)(0) \neq 0$;
  \item $D\xi(x)>0$ for all $\gamma_i < x< \gamma_{i+1}$, $i \in \{0, 1, \cdots, N_1 -1\}$ and $D\eta(x)>0$ for all
  $\beta_j < x <\beta_{j+1}-1$,  $j \in \{ 0, 1, \cdots, N_2-1 \}$;
  \item The origin has the same criticality for $\eta$ than for $\xi$;
  \item For each $1 \leq k \leq r$, we have that $D^{k}_{-}(\xi \circ \eta)(0)=D^{k}_{+}(\eta \circ \xi)(0)$, where $D_{-}^k$ and $D_{+}^k$ represent the $k$-th left and right derivative, respectively.
 \end{enumerate}
\end{definition}

\begin{figure}[ht]
 \centering
  \psfrag{n}[][]{$\eta$} 
  \psfrag{e}[][]{$\xi$}
 \psfrag{0}[][]{$0$}
  \includegraphics[width=2.3in]{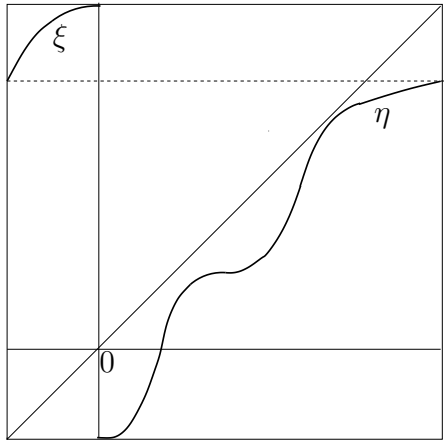}
  \caption{A bicritical commuting pair $\zeta=(\eta,\xi)$.}
\end{figure}

Let $\zeta_1=(\eta_{1}, \xi_{1})$ and $\zeta_2=(\eta_2, \xi_2)$ be two $C^r$ multicritical commuting pairs, and let
$\tau_1: [\eta_1(0), \xi_1(0)] \rightarrow [-1,1]$ and $\tau_2: [\eta_2(0), \xi_2(0)] \rightarrow [-1,1]$ be the two M\"obius transformations given by$$ \tau_{i}(\eta_{i}(0))=-1, \ \ \tau_{i}(0)=0 \ \text{ and } \ \tau_{i}(\xi_{i}(0))=1,\quad\mbox{for each $i \in \{1,2\}$.}$$

\begin{definition}\label{pseudometric} For any given $0 \leq k \leq r$ we define the $C^k$ distance between $\zeta_1$ and $\zeta_2$ as
\begin{equation*}
d_{k}(\zeta_1,\zeta_2)= \max \left\lbrace \left| \dfrac{\eta_1(0)}{\xi_1(0)} - \dfrac{\eta_2(0)}{\xi_2(0)} \right| ,  \
\|\tau_{1}\circ \zeta_1 \circ \tau_1^{-1} - \tau_2 \circ \zeta_2 \circ \tau_{2}^{-1} \|_{k} \right\rbrace
\end{equation*}
where $\| \cdot \|_{k}$ denotes the $C^{k}$ norm for maps in the interval $[-1,1]$ with a discontinuity at the origin.
\end{definition}

Note that $d_{k}(,)$ is not a distance but a pseudo-distance, since it is invariant under conjugacies with homothetias. In order to have a distance we restrict our attention to \emph{normalized} pairs: for any given pair $\zeta=(\eta,\xi)$ we denote by $\widetilde{\zeta}$ the pair $(\widetilde{\eta}|_{\widetilde{I_{\eta}}}, \widetilde{\xi}|_{\widetilde{I_{\xi}}})$, where tilde means linear rescaling by the factor $1/|I_{\eta}|$. In other words, $|\widetilde{I_{\eta}}|=1$ and $\widetilde{I_{\xi}}$ has length equal to the ratio between the lengths of $I_{\xi}$ and $I_{\eta}$. Equivalently, $\widetilde{\xi}(0)=1$ and $\widetilde{\eta}(0)=-|I_{\xi}|/|I_{\eta}|=\eta(0)/\xi(0)$.

\subsection{Renormalization of multicritical commuting pairs}\label{rmcm}

\begin{definition} We define the \emph{period} of the pair $\zeta=(\eta, \xi)$ as the natural number $a$ such that$$ \eta^{a+1}(\xi(0)) <0\leq \eta^{a}(\xi(0)),$$when such number exists, and we denote it by $\chi(\zeta)$. If such $a$ does not exist, we just define $\chi(\zeta)=\infty$.
\end{definition}
 
\begin{definition}\label{defren} Let $\zeta=(\eta, \xi)$ be a multicritical commuting pair with $(\xi \circ \eta)(0) \in I_{\eta}$ and $\chi(\zeta)=a< \infty$. We define the \emph{renormalization} of $\zeta$ as the normalization of the pair $(\eta|_{[0, \eta^{a}(\xi(0)) ]} \ , \ \eta^{a}\circ \xi|_{I_{\xi}} )$, that is:$$\mathcal{R}(\zeta)= \left(\widetilde{\eta}|_{[0,\widetilde{\eta^{a}(\xi(0))} ]} \ , \ \widetilde{\eta^{a}\circ \xi}|_{\widetilde{I}_{\xi}}
  \right).$$
\end{definition}
 
If $\zeta$ is a multicritical commuting pair with $\chi(\mathcal{R}^{j}\zeta)< \infty$ for $0 \leq j \leq n-1$, we say that $\zeta$ is
 \textit{$n$-times renormalizable}, and if $\chi(\mathcal{R}^{j}\zeta)< \infty$  for all $j \in \nt$, we say that $\zeta$ is \textit{infinitely renormalizable}. In the last case, we define the \textit{rotation number} of $\zeta$ as the irrational number whose continued fraction expansion is given by$$[\chi(\zeta), \chi(\mathcal{R}\zeta), \cdots, \chi(\mathcal{R}^{n}\zeta), \cdots ].$$
  
Now let $f$ be a $C^r$ multicritical circle map with irrational rotation number $\rho$ and $N$ critical points $c_{0}, \dots, c_{N-1}$. For each critical point $c_i$, $f$ induces a sequence of multicritical commuting pairs in the following way: let $\widehat{f}$ be the lift of $f$ (under the universal covering $t \mapsto c_i\cdot\exp(2\pi i t)$) such that $0< \widehat{f}(0)<1$ (and note that $D\widehat{f}(0)=0$). For $n\geq 1$, let $\widehat{I}_{n}(c_i)$ be the closed interval in $\R$, containing the origin as one of its extreme points, which is projected onto $I_{n}(c_i)$. We define $\xi: \widehat{I}_{n+1}(c_i) \rightarrow \R$ and $\eta: \widehat{I}_{n}(c_i) \rightarrow \R$ by $\xi= T^{-p_{n}}\circ \widehat{f}^{q_{n}}$ and $\eta= T^{-p_{n+1}}\circ \widehat{f}^{q_{n+1}}$, where $T$ is the unitary translation $T(x)=x+1$. Then the pair $(\eta|_{\widehat{I}_{n}(c_i)}, \xi|_{\widehat{I}_{n+1}(c_i)})$ is an infinitely renormalizable multicritical commuting pair, that we denote by $(f^{q_{n+1}}|_{I_n(c_i)}, f^{q_n}|_{I_{n+1}(c_i)})$. Its normalization will be denoted by $\mathcal{R}_i^{n}f$, that is:$$\mathcal{R}_i^{n}f=\left(\widetilde{f}^{q_{n+1}}|_{\widetilde{I_n}(c_i)}, \widetilde{f}^{q_{n}}|_{\widetilde{I_{n+1}}(c_i)}\right).$$
  
\begin{figure}[!ht]
 \centering
\psfrag{F}[][]{$f^{q_{n+1}}$}
\psfrag{G}[][]{$ \ \ \ f^{q_n}$}
\psfrag{H}[][]{$ \ \ f^{q_{n+2}}$}
\includegraphics[width=2.6in]{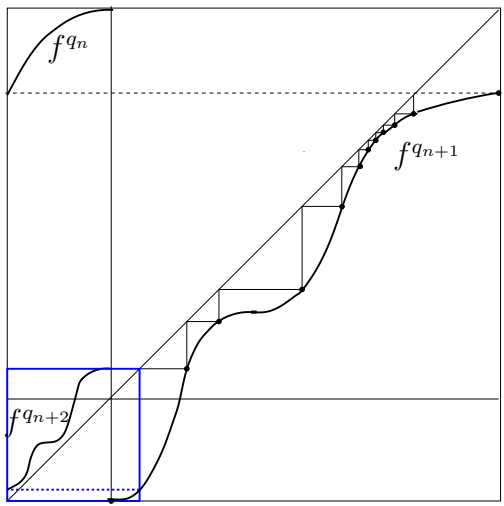}
\caption{The $n-$th and $(n+1)-$th renormalization of $f$, before rescaling.}
\end{figure}

\section{A reduction of Theorem \ref{maintheorem}}\label{secreduction} In this section we reduce our main result, namely Theorem \ref{maintheorem}, to Theorem \ref{theoremb} below, which is slightly easier to prove. Right after its statement, we explain why Theorem \ref{theoremb} implies Theorem \ref{maintheorem}.

\begin{maintheorem}\label{theoremb} Let $f$ and $g$ be $C^3$ bicritical circle maps with the same irrational rotation number in the set $\mathcal{A}$. Suppose that both $f$ and $g$ have the same signature and exactly the same critical set (in other words, there exists a topological conjugacy $h$ fixing each critical point). Assume, finally, that there exist $C>1$ and $0<\mu<1$ such that for each $c_i \in \crit(f)$ we have$$\left| \frac{|I_{n}^g(c_i)|}{|I_{n}^f(c_i)|}- 1 \right| \leq C\, \mu^{n} \quad\mbox{and}\quad d_{1}(\mathcal{R}_i^nf, \mathcal{R}_i^ng) \leq C\, \mu^{n}.$$

Then $h$ is a $C^{1+\alpha}$ diffeomorphism for some $\alpha>0$.
\end{maintheorem}

Let us briefly explain why Theorem \ref{theoremb} implies Theorem \ref{maintheorem}. First we note that, as pointed out in \cite[Proposition 2.2]{dFdM}, the real bounds (Theorem \ref{teobeau}) imply that exponential convergence of renormalization is preserved under conjugacy with a smooth diffeomorphism. Let us be more precise.

\begin{lemma}\label{lema0} Let $r\geq 1$, $f$ a $C^r$ multicritical circle map and $\phi$ a $C^r$ circle diffeomorphism. There exist $C=C(f, \phi)>0$ and $0<\mu=\mu(f)<1$ such that, for all $k \leq r-1$ and all $n \in \nt$, we have
\[
d_k \left(\mathcal{R}^{n}_i f \ , \ \mathcal{R}^{n}_{i}(\phi \circ f \circ \phi^{-1}) \right) \leq C\, \mu^{n},
\]for any given critical point $c_i$ of $f$, where $\mathcal{R}^{n}_{i}(\phi \circ f \circ \phi^{-1})$ denotes the $n$-th renormalization of $\phi \circ f \circ \phi^{-1}$ around its critical point $\phi(c_i)$.
\end{lemma}

The following result is borrowed from \cite[Lemma 4.7]{dFdM}.
 
\begin{lemma}\label{corolemma4.7} Let $f$ and $g$ be two multicritical circle maps with the same critical set and such that there exist $C>0$ and $0<\mu<1$ satisfying $d_0(\mathcal{R}_i^{n}f, \mathcal{R}_i^{n}g) \leq C\, \mu^{n}$ for all $i\in\{0,...,N-1\}$ and for all $n\in \nt$. Then the ratio $\big\{|I_{n}^g(c_i)|/|I_{n}^f(c_i)|\big\}$ converges to a limit exponentially fast for all $i\in\{0,...,N-1\}$. Moreover, for all $m,k \geq 1$  we have:
\begin{equation}\label{estratios}
\left| \frac{|I_{m}^f(c_i)|}{|I_{k}^f(c_i)|}- \frac{|I_{m}^g(c_i)|}{|I_{k}^g(c_i)|} \right| \leq C\, \mu^{ min\{m,k \}} 
\frac{|I_{m}^f(c_i)|}{|I_{k}^f(c_i)|}\,.
\end{equation}
\end{lemma}

We remark that estimate \eqref{estratios} given by Lemma \ref{corolemma4.7} will also be useful in Section \ref{SecProofThmB}, during the proof of Lemma \ref{keylemma1}. With Lemma \ref{lema0} and Lemma \ref{corolemma4.7} at hand, it is not difficult to see that Theorem \ref{theoremb} implies Theorem \ref{maintheorem}. Indeed, let $f$ and $g$ be $C^3$ bicritical circle maps with the same signature and such that its common rotation number belongs to the set $\mathcal{A}$. Assume, moreover, that the renormalizations of $f$ and $g$ around corresponding critical points converge together exponentially fast in the $C^1$ topology. By Lemma \ref{lema0} we can conjugate one of the two maps (say $g$) with a suitable $C^{\infty}$ diffeomorphism that identifies the critical points of $g$ with those of $f$, while preserving the exponential contraction in the $C^1$ metric. By Lemma \ref{corolemma4.7}, we can choose the previous conjugacy in such a way that the limit of the sequence $\big\{|I_{n}^g(c_i)|/|I_{n}^f(c_i)|\big\}$ is in fact equal to $1$, for all $i\in\{0,...,N-1\}$. By Theorem \ref{theoremb}, $f$ and $g$ are conjugate to each other by a $C^{1+\alpha}$ diffeomorphism, for some $\alpha>0$. This shows that Theorem \ref{theoremb} implies Theorem \ref{maintheorem}.

Sections \ref{sec:anewpartition} and \ref{SecProofThmB} (the remainder of this paper) are devoted to the proof of Theorem \ref{theoremb}.

\section{Fine grids}\label{sec:anewpartition}

Let $f$ and $g$ be $C^3$ bicritical circle maps with the same irrational rotation number in the set $\mathcal{A}$ (recall Definition \ref{setA}). Suppose that both $f$ and $g$ have the same signature and exactly the same critical set, and let $h$ be the homeomorphism considered in the statement of Theorem \ref{theoremb} (Section \ref{secreduction} above). As explained in the introduction, we would like to prove that $h$ ``almost preserves" ratios between lengths of intervals, provided we consider very small intervals, which are very close to each other. To achieve this, we will construct in this section a suitable sequence of nested partitions of the unit circle, such that it will be enough to control the action of $h$ on the vertices of those partitions. The specific type of partitions that we need in the present paper are given by the following definition, which is borrowed from \cite[Section 4.2]{dFdM}.

\begin{definition}\label{def:finegrid} A \emph{fine grid} is a sequence $\{\mathcal{Q}_{n}\}_{n \geq 0}$ of finite interval partitions of $S^{1}$ satisfying the following three conditions. 
\begin{enumerate}
\item\label{item1FG} Each $\mathcal{Q}_{n+1}$ is a strict refinement of $\mathcal{Q}_n$;
\item\label{item2FG} There exists $b\in\nt$ such that each atom of $\mathcal{Q}_{n}$ coincides with the union of at most $b$ atoms of $\mathcal{Q}_{n+1}$;
\item\label{item3FG} There exists $C >1$ such that $C^{-1}\,|I| \leq |J| \leq C\,|I|$ for each pair of adjacent atoms $I,J\in\mathcal{Q}_{n}$. 
\end{enumerate}
\end{definition}

The fundamental property of fine grids that we will use here is the following criterion, which is \cite[Proposition 4.3(b)]{dFdM}.

\begin{proposition}\label{criterion} Let $\{\mathcal{Q}_{n}\}_{n\in\nt}$ be a fine grid as in Definition \ref{def:finegrid}. Let $h$ be a circle homeomorphism 
such that there exist constants $C>0$ and $\lambda \in (0,1)$ satisfying
\begin{equation}\label{Carlesonineq}
\left|\dfrac{|I|}{|J|}- \frac{|h(I)|}{|h(J)|} \right| \leq C\lambda^{n},
\end{equation}
for each pair of adjacent intervals $I,J \in \mathcal{Q}_{n}$ and for all $n\in\nt$. Then $h$ is a $C^{1+\alpha}$-diffeomorphism.
\end{proposition} 

\begin{proof}[Sketch of the proof of Proposition \ref{criterion}] As it easily follows from Definition \ref{def:finegrid}, given a fine grid $\{\mathcal{Q}_{n}\}_{n\in\nt}$ there exist constants $C_0>1$ and $0<\lambda_0<\lambda_1<1$ such that
\begin{equation}\label{lema4.2dFdM}
\frac{1}{C_0}\,\lambda_0^n\leq\min_{I\in\mathcal{Q}_n}\big\{|I|\big\}\leq\max_{I\in\mathcal{Q}_n}\big\{|I|\big\}\leq C_0\,\lambda_1^n\quad\mbox{for all $n\in\nt$.}
\end{equation}Moreover, an immediate consequence of condition \eqref{Carlesonineq} in the statement is that the image under $h$ of the fine grid $\{\mathcal{Q}_{n}\}_{n\in\nt}$ is also a fine grid. From this and \eqref{lema4.2dFdM}, we deduce that the sequence of piecewise affine homeomorphisms that coincide with $h$ on the vertices of each $\mathcal{Q}_{n}$ converges uniformly ($C^0$ exponentially fast) to $h$. Since each fine grid is determined by a finite number of vertices, these approximations have a well defined right-derivative (which is a step function with finitely many jumps), and it can be proved (combining Definition \ref{def:finegrid}, \eqref{Carlesonineq} and \eqref{lema4.2dFdM}) that these right-derivatives converge uniformly to an $\alpha$-H\"older continuous function (whose H\"older constant $\alpha$ depends on $\lambda$ and $\lambda_0$). By elementary reasons, this implies that $h$ is a $C^{1+\alpha}$-diffeomorphism. For more details, see \cite[pages 357-358]{dFdM}.
\end{proof}

We remark that the standard partitions $\mathcal{P}_n$ (see Section \ref{secrealbounds}) do not determine a fine grid, unless the rotation number of $f$ is of bounded type. Our goal in this section is to construct a suitable fine grid for any given bicritical circle map $f$, while in Section \ref{SecProofThmB} we will prove that such fine grid (together with the topological conjugacy $h$ considered in Theorem \ref{theoremb}) satisfies the assumptions of Proposition \ref{criterion}. This will establish Theorem \ref{theoremb}. As already explained in Section \ref{secreduction}, Theorem \ref{theoremb} implies our main result, namely Theorem~\ref{maintheorem}.

\subsection{Auxiliary partitions}\label{secinter} Let $f$ be a $C^3$ bicritical circle map with irrational rotation number $\rho\in(0,1)$ and critical points $c_0$ and $c_1$ (we will focus now on $c_0$, but of course all constructions below can be done with $c_1$). As explained in Section \ref{rmcm}, for any given $n\in\nt$, the first return map of $f$ to $J_n(c_0)=I_{n+1}(c_0) \cup I_{n}(c_0)$ is given by the commuting pair $\big(f^{q_{n+1}}|_{I_n(c_0)}, f^{q_n}|_{I_{n+1}(c_0)}\big)$. This return map has two critical points as well: one of them being $c_0$ itself, and the other one being the unique preimage of $c_1$ for the return (note that they coincide if, and only if, $c_1$ belongs to the positive orbit of $c_0$). Such a critical point for the return map will be called \emph{the free critical point at level $n$}, and it will be denoted by~$\mathfrak{c}_{n}$.

\begin{definition}\label{deftwobridgeslevel} A natural number $n$ is a \emph{two-bridges level} for $f$ at $c_0$ if $a_{n+1} \geq 23$, the free critical point $\mathfrak{c}_{n}$ belongs to $I_{n}(c_0)\setminus I_{n+2}(c_0)$ and moreover$$\mathfrak{c}_{n}\in\bigcup_{j=11}^{a_{n+1}-10}\Delta_j\,,$$where $\Delta_j=f^{(j-1)q_{n+1}+q_n}\big(I_{n+1}(c_0)\big)$ for all $j\in\{1,...,a_{n+1}\}$.
\end{definition}

\begin{remark} Of course there is nothing special about the number $23$. It is just an arbitrary choice that we fix throughout the remainder of this paper.
\end{remark}

In this subsection we construct a sequence $\big\{\widehat{\mathcal{P}}_n\big\}_{n\in\nt}$ of finite interval partitions (modulo endpoints) of the unit circle, satisfying the following six properties.
\begin{enumerate}
\item\label{itemdyndef} Each partition $\widehat{\mathcal{P}}_n$ is dynamically defined from the critical set of $f$: all its vertices are iterates (either forward or backward) of $c_0$ or $c_1$.
\vspace{.15cm}
\item\label{itemfundom} Both intervals $I_{n}(c_0)$ and $I_{n+1}(c_0)$ belong to $\widehat{\mathcal{P}}_{n}$.
\vspace{.15cm}
\item\label{itemref} The partition $\widehat{\mathcal{P}}_{n+1}$ is a \emph{refinement} of $\widehat{\mathcal{P}}_{n}$\,: each interval of $\widehat{\mathcal{P}}_{n}$ either coincides with the disjoint union of at least two intervals of $\widehat{\mathcal{P}}_{n+1}$, or belongs itself to $\widehat{\mathcal{P}}_{n+1}$ (in which case it coincides with the disjoint union of at least two intervals of $\widehat{\mathcal{P}}_{n+2}$).
\vspace{.15cm}
\item\label{itemfree} If $n$ is a two-bridges level for $f$ at $c_0$, the free critical point $\mathfrak{c}_{n}$ is a vertex of $\widehat{\mathcal{P}}_{n+1}$.
\vspace{.15cm}
\item\label{itemreappear} Any vertex of the standard partition $\mathcal{P}_n$ belongs to $\widehat{\mathcal{P}}_{m}$ for some $m \geq n$.
\vspace{.15cm}
\item\label{itemboundgeom} There exists a constant $C>1$ (depending only on $f$) such that $C^{-1}\,|I| \leq |J| \leq C\,|I|$ for each pair of adjacent atoms $I,J\in\widehat{\mathcal{P}}_n$.
\end{enumerate}

The first five properties above describe the \emph{combinatorics} of the sequence $\{\widehat{\mathcal{P}}_n\}$, while Item \eqref{itemboundgeom} bounds its \emph{geometry}. The main difference between the partitions $\widehat{\mathcal{P}}_n$ and the standard partitions $\mathcal{P}_n(c_0)$ is Item \eqref{itemfree}. The partitions $\widehat{\mathcal{P}}_n$ will be called \emph{auxiliary partitions} around $c_0$. Just as the standard partitions, they do not determine a fine grid, unless the rotation number of $f$ is of bounded type. However, in Section \ref{sec:constructingfinegrid} we will use these auxiliary partitions to finally build a fine grid $\{\mathcal{Q}_n\}_{n\in\nt}$ for the bicritical circle map $f$ (this fine grid will be extracted from the auxiliary partitions, in the sense that any vertex of $\mathcal{Q}_n$ will also be a vertex of $\widehat{\mathcal{P}}_n$, see Proposition \ref{gridpprop} in Section \ref{secfinalFG} below). Further properties of the auxiliary partitions (such as Lemma \ref{lematrayectoria} below) will be useful in Section \ref{SecProofThmB}.

\subsubsection{Building auxiliary partitions}\label{secbuildaux} For the initial partition $\widehat{\mathcal{P}}_{0}$ we simply consider the standard partition $\mathcal{P}_0(c_0)$, that is:$$\widehat{\mathcal{P}}_{0}=\mathcal{P}_0(c_0)=\left\{\big[f^{i}(c_0),f^{i+1}(c_0)\big]:\,i\in\{0,...,a_0-1\}\right\}\cup\big\{\big[f^{a_{0}}(c_0),c_0\big]\big\},$$where $a_0$ is the integer part of $1/\rho$ (see Section \ref{secrealbounds}). Now we fix some $n\in\nt$ and we build $\widehat{\mathcal{P}}_{n+1}$ from $\widehat{\mathcal{P}}_{n}$ (this defines inductively the whole sequence $\{ \widehat{\mathcal{P}}_{n}\}_{n\in\nt}$).

On one hand, if $n$ is a two-bridges level for $f$ at $c_0$, consider the following three pairwise disjoint fundamental domains for $f^{q_{n+1}}$, all of them contained in $I_{n}(c_0)\setminus I_{n+2}(c_0)$:
\begin{align*}
&\Delta_1=f^{q_n}\big(I_{n+1}(c_0)\big)\,,\\
&\widehat{\Delta}_n=[f^{q_{n+1}}(\mathfrak{c}_{n}), \mathfrak{c}_{n}]\,,\\
&\Delta_{a_{n+1}}=f^{(a_{n+1}-1)q_{n+1}}(\Delta_1)=\big[f^{q_{n+2}}(c_0),f^{(a_{n+1}-1)q_{n+1}+q_n}(c_0)\big]\,.
\end{align*}For each $j \in \mathbb{Z}$\,, denote by $\Delta_1^j$\,,\, $\widehat{\Delta}_n^{j}$\, and\, $\Delta_{a_{n+1}}^j$ the intervals $f^{jq_{n+1}}(\Delta_1)$, $f^{jq_{n+1}}(\widehat{\Delta}_n)$ and $f^{jq_{n+1}}(\Delta_{a_{n+1}})$ respectively. Let $r(n)\,,\,\ell(n)\in\{0, \dots, a_{n+1}\}$ be given by:\[r(n)= \min \{ j \in \nt: \widehat{\Delta}_n^{-j} \cap \Delta_1^j \neq \emptyset  \} \hspace{0.4cm} \text{and}\hspace{0.4cm} \ell(n)= \min \{ j \in \nt: \Delta_{a_{n+1}}^{-j} \cap \widehat{\Delta}_n^{j}  \neq \emptyset  \}.\]Note that the intersections above may be given by a single point. Just to fix ideas, let us assume that $\Delta_{a_{n+1}}^{-\ell}\setminus\widehat{\Delta}_n^{\ell}\subseteq\widehat{\Delta}_n^{\ell+1}$ and $\widehat{\Delta}_n^{-r}\setminus\Delta_1^r\subseteq\Delta_1^{r+1}$, and consider $\Delta_{R}^{n+1}=\Delta_1^r \cup \widehat{\Delta}_n^{-r}$ and $\Delta_{L}^{n+1}=\Delta_{a_{n+1}}^{-\ell} \cup  \widehat{\Delta}_n^{\ell}$. With this at hand, we define the auxiliary partition $\widehat{\mathcal{P}}_{n+1}$ inside $I_n(c_0)$, for a two-bridges level $n$, as$$\widehat{\mathcal{P}}_{n+1}|_{I_n(c_0)}=\big\{I_{n+2}(c_0)\,,\,\{\Delta_{a_{n+1}}^{-j}\}_{j=0}^{j=\ell-1},\, \Delta_{L}^{n+1}\,,\,\{\widehat{\Delta}_n^{j}\}_{j=1-r}^{\ell-1}\,,\,\Delta_{R}^{n+1}\,,\,\{\Delta_1^j\}_{j=0}^{j=r-1}\big\}\,,$$and we spread this definition to the whole circle in the usual way:$$\widehat{\mathcal{P}}_{n+1}=\left\{f^{i}\big(I_{n+1}(c_0)\big):\,0 \leq i \leq q_{n}-1\right\}\cup\left\{ f^{j}(I):I\in\widehat{\mathcal{P}}_{n+1}|_{I_n(c_0)}\,,\,0\leq j\leq q_{n+1}-1 \right\}.$$

\begin{figure}[!ht]
\centering
\psfrag{In+2}[][][1]{$I_{n+2}$} 
\psfrag{Deltaan+1}[][][1]{$\Delta_{a_{n+1}}$}
\psfrag{R}[][][1]{$\Delta_{a_{n+1}}^{-\ell+1}$}
\psfrag{Deltal}[][]{$\Delta_{L}^{n+1}$} 
\psfrag{S}[][][1]{$\small{\widehat{\Delta}_n^{\ell-1}}$}
\psfrag{F}[][][1]{$f^{q_{n+1}}$}
\psfrag{G}[][][1]{$f^{-q_{n+1}}$} 
\psfrag{Delta0}[][][1]{$\Delta_1$}
\psfrag{c}[][][1]{$\mathfrak{c}_{n}$}
\psfrag{Delta0}[][][1]{$\Delta_1$}
\psfrag{Deltahat}[][][1]{$\widehat{\Delta}_n$}
\psfrag{Deltahat-1}[][][1]{$\widehat{\Delta}_n^{-1}$}
\psfrag{V}[][][1]{$\widehat{\Delta}_n^{-1}$}
\psfrag{T}[][][1]{$\small{\widehat{\Delta}_n^{-r+1}}$}
\psfrag{Deltar}[][]{$\Delta_{R}^{n+1}$}
\psfrag{U}[][]{$\Delta_1^{r-1}$} 
\psfrag{...}[][]{$\dots$} 
\includegraphics[width=6.3in]{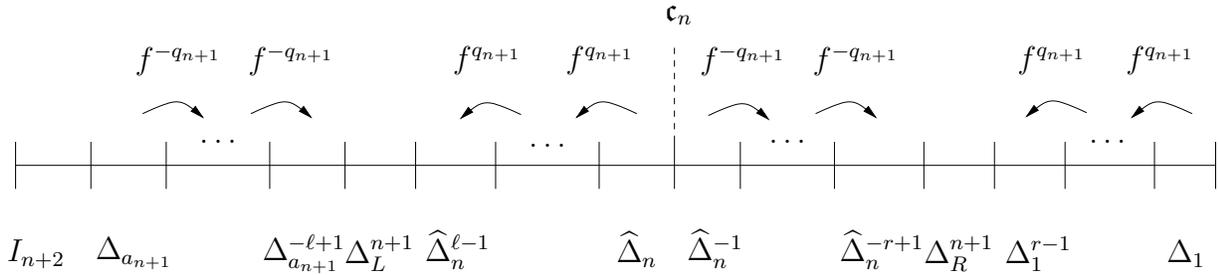}
\caption{\label{decomposition} The auxiliary partition $\widehat{\mathcal{P}}_{n+1}$ inside $I_n(c_0)$, for a two-bridges level $n$.}
\end{figure}

On the other hand, if $n$ is \emph{not} a two-bridges level for $f$ at $c_0$, we would like to consider $\widehat{\mathcal{P}}_{n+1}$ just as the standard partition $\mathcal{P}_{n+1}(c_0)$. However, $\mathcal{P}_{n+1}(c_0)$ is given only by iterates of $c_0$, and then it does not have $\mathfrak{c}_{i}$ as a vertex for any two-bridges level $i\in\{0,...,n-1\}$. In other words, the partition $\mathcal{P}_{n+1}(c_0)$ is not a refinement of $\widehat{\mathcal{P}}_{n}$ (unless, of course, no previous level was a two-bridges level). To correct this flaw, we simply proceed as follows: for any given vertex $v$ of $\widehat{\mathcal{P}}_{n}$, let $w$ be the vertex of $\mathcal{P}_{n+1}(c_0)$ closest to $v$ (in the Euclidean distance). Then we just replace $w$ by $v$: we remove $w$ from $\mathcal{P}_{n+1}(c_0)$, and we add $v$ to this partition (in case $v$ is the middle point of the interval of $\mathcal{P}_{n+1}(c_0)$ that contains it, we just add it to $\mathcal{P}_{n+1}(c_0)$ without removing any point). After this \emph{replacement procedure}, we denote by $\widehat{\mathcal{P}}_{n+1}$ the obtained partition.

With this construction at hand, we define inductively the sequence of finite partitions $\big\{\widehat{\mathcal{P}}_{n}\big\}_{n\in\nt}$ of the unit circle. The combinatorial properties \eqref{itemdyndef} to \eqref{itemreappear} listed above are not difficult to check, while Item \eqref{itemboundgeom} follows by combining the following lemma with the real bounds (Theorem \ref{teobeau}).

\begin{lemma}\label{lemaintersect} If $\Delta\in\widehat{\mathcal{P}}_n$ and $\Delta'\in \mathcal{P}_n(c_0)$ are two atoms such that $\Delta\cap\Delta'\neq\emptyset$, then $|\Delta|\asymp|\Delta'|$.
\end{lemma}

\begin{proof}[Proof of Lemma \ref{lemaintersect}] We fix some $n$ and we prove the desired comparability for intersecting atoms of $\widehat{\mathcal{P}}_{n+1}$ and $\mathcal{P}_{n+1}(c_0)$ respectively. For a two-bridges level $n$ (for $f$ at $c_0$), we have three different types of atoms of $\widehat{\mathcal{P}}_{n+1}$:
\begin{itemize}
\item The following atoms of $\widehat{\mathcal{P}}_{n+1}$ also belong to $\mathcal{P}_{n+1}(c_0)$:
\begin{align*}
&f^{i}\big(I_{n+1}(c_0)\big):\,0 \leq i \leq q_{n}-1\,,\\
&f^{i}\big(I_{n+2}(c_0)\big):\,0 \leq i \leq q_{n+1}-1\,,\\
&f^{i}\big(\Delta_{a_{n+1}}^{-j}\big):\,0 \leq j \leq \ell-1\,,\,0 \leq i \leq q_{n+1}-1\,,\\
&f^{i}\big(\Delta_1^j\big):\,0 \leq j \leq r-1\,,\,0 \leq i \leq q_{n+1}-1\,.
\end{align*}
\item For the atoms $f^{i}(\Delta_{L}^{n+1})$ and $f^{i}(\Delta_{R}^{n+1})$, with $0 \leq i \leq q_{n+1}-1$, we just note the following: both $f^{i}\big(\Delta_{a_{n+1}}^{-\ell}\big)$ and $f^{i}\big(\Delta_{1}^{r}\big)$ belong to $\mathcal{P}_{n+1}(c_0)$ for any $0 \leq i \leq q_{n+1}-1$, and then we can apply Remark \ref{remcomp} with $I=f^{i}\big(\Delta_{a_{n+1}}^{-\ell}\big)$ and $J=f^{i}(\Delta_{L}^{n+1})$, and also with $I=f^{i}\big(\Delta_{1}^{r}\big)$ and $J=f^{i}(\Delta_{R}^{n+1})$.
\item For any $1-r \leq j \leq \ell-1$ and any $0 \leq i \leq q_{n+1}-1$, the interval $f^{i}(\widehat{\Delta}_n^{j})$ intersects at most two atoms of $\mathcal{P}_{n+1}(c_0)$, say $I$ and $J$, both being consecutive fundamental domains of $f^{q_{n+1}}$. Since $f^{i}(\widehat{\Delta}_n^{j}) \subseteq I \cup J \subseteq f^{i}(\widehat{\Delta}_n^{j+1}) \cup f^{i}(\widehat{\Delta}_n^{j}) \cup f^{i}(\widehat{\Delta}_n^{j-1})$, we are done by Corollary \ref{symmetricintervals}.
\end{itemize}Therefore, we have comparability when $n$ is a two-bridges level. Finally, by the real bounds, if $n$ is \emph{not} a two-bridges level, then the replacement procedure described above (to build the auxiliary partition $\widehat{\mathcal{P}}_{n+1}$ from the standard partition $\mathcal{P}_{n+1}(c_0)$) creates neither small nor big intervals since, given a missing vertex of $\widehat{\mathcal{P}}_{n}$, we remove from $\mathcal{P}_{n+1}(c_0)$ its closest vertex. Therefore, we also have comparability when $n$ is not a two-bridges level.
\end{proof}

As proved in \cite[Lemma 4.1]{EdF}, any two intersecting atoms belonging to the same level of the standard dynamical partitions of two distinct critical points are comparable. When combined with Lemma \ref{lemaintersect}, this gives us the following fact, that will be mentioned in Section \ref{SecProofThmB} (during the proof of Lemma \ref{ineqbeta}).

\begin{corollary}\label{corointersect} If $\Delta\in\widehat{\mathcal{P}}_n$ and $\Delta'\in \mathcal{P}_n(c_1)$ are two atoms such that $\Delta\cap\Delta'\neq\emptyset$, then $|\Delta|\asymp|\Delta'|$.
\end{corollary}

In Section \ref{SecProofThmB} we will also use the following immediate consequence of properties \eqref{itemref} and \eqref{itemboundgeom} of the auxiliary partitions.

\begin{corollary}\label{mu3} There exists $\mu\in(0,1)$ such that $|J|\leq\mu|I|$ for all $I\in\widehat{\mathcal{P}}_{n}$ and $J\in\widehat{\mathcal{P}}_{n+1}$ with $J \subsetneq I$.
\end{corollary}

\subsubsection{A combinatorial remark} We finish Section \ref{secinter} with Lemma \ref{lematrayectoria} below, which is an adaptation of \cite[Lemma 4.9]{dFdM} to the auxiliary partitions, that is going to be crucial in Section \ref{SecProofThmB}, during the proof of Lemma \ref{keylemma1}. Let us point out first the following fact, that follows straightforward from our construction.

\begin{lemma} Let $n\in\nt$ and $y \in J_n(c_0) \setminus J_{n+1}(c_0)$. Then the following holds.
\begin{itemize}
\item If\, $n$ is not a two-bridges level, there exist $x \in J_{n+1}(c_0)$, $\sigma\in\{0,1\}$ and $k\in\Z$, with $|k| \leq \lceil a_{n+1}/2 \rceil$, such that $y=f^{kq_{n+1}+\sigma q_{n}}(x)$.
\item If\, $n$ is a two-bridges level, there exist $x \in J_{n+1}(c_0)\cup\widehat{\Delta}_n$, $\sigma\in\{0,1\}$ and $k\in\Z$, with $|k|\leq\ell(n)$ and $|k| \leq r(n)$, such that $y=f^{kq_{n+1}+\sigma q_{n}}(x)$.
\end{itemize}
\end{lemma}

By induction we obtain the following description (for more details, see \cite[pages 363-364]{dFdM}).

\begin{lemma}\label{lematrayectoria} Let $n,p \in \nt$, and let $v$ be a vertex of $\widehat{\mathcal{P}}_{n+p}$ contained in $J_n(c_0)$. Then there exist $L\in\{1,...,p\}$ and $n \leq m_1 < ... < m_L \leq n+p$ such that\, $v = \varphi_1 \circ \cdots \circ \varphi_L(x)$\,, where:
\begin{itemize}
\item For each $j\in\{1,...,L\}$ we have $\varphi_j= f^{k_jq_{m_j+1}+\sigma_jq_{m_j}}$ for some $\sigma_j\in\{0,1\}$ and $k_j\in\Z$, where each $k_j$ either satisfies $|k_j|\leq\ell(m_j)$ and $|k_j| \leq r(m_j)$ or $|k_j| \leq \lceil a_{m_j+1}/2 \rceil$, depending on whether $m_j$ is or is not a two-bridges level (for $f$ at $c_0$).
\item For each $j\in\{1, ... , L-1\}$, the point $\varphi_{j+1} \circ \cdots \circ \varphi_L(x)$ either belongs to $J_{m_j+1}(c_0)\cup\widehat{\Delta}_{m_j}$ or to $J_{m_j+1}(c_0)$, depending on whether $m_j$ is or is not a two-bridges level.
\item There exists $m\in\{m_L,...,n+p\}$ such that the initial condition $x$ either belongs to $\big\{c_0\,,\,f^{q_{m+2}}(c_0)\,,\,\mathfrak{c}_{m}\big\}$ or to $\big\{c_0\,,\,f^{q_{m+2}}(c_0)\big\}$, depending on whether $m$ is or is not a two-bridges level.
\end{itemize}
\end{lemma}

As already mentioned, the auxiliary partitions do not determine a fine grid, unless the rotation number of $f$ is of bounded type. In the next subsection we finally construct a fine grid for $f$.

\subsection{Building a fine grid}\label{sec:constructingfinegrid} In the remainder of Section \ref{sec:anewpartition} we adapt the construction of \cite[Section 4.3]{dFdM} to our setting, following the exposition in \cite{EdF}. More precisely, in sections \ref{subsecinter}, \ref{subsecalm} and \ref{secbalanponte} we follow \cite[Sections 4.4\,--\,4.6]{EdF}, while in section \ref{secfinalFG} we follow \cite[Section 5.2]{EdF}.

\subsubsection{Intermediate partitions}\label{subsecinter}

\begin{definition}\label{defbridge} We define \emph{bridges} for the auxiliary partitions as follows.
\begin{itemize}
\item If\, $n$ is a two-bridges level for $f$ at $c_0$ (see Definition \ref{deftwobridgeslevel}), then the bridges of\, $\widehat{\mathcal{P}}_{n+1}|_{I_n(c_0)}$  are the intervals $\widehat{G}_1$ and $\widehat{G}_2$ given by$$\widehat{G}_1= \bigcup_{j=1}^{r-1} \Delta_1^j  \, \cup \, \Delta_{R}^{n+1} \, \cup  \bigcup_{j=2}^{r-1}\widehat{\Delta}_n^{-j}\quad\mbox{and}\quad\widehat{G}_2= \bigcup_{j=1}^{\ell-1} \widehat{\Delta}_n^{j}  \, \cup \, \Delta_{L}^{n+1} \, \cup  \bigcup_{j=0}^{\ell-1}\Delta^{-j}_{a_{n+1}}\,.$$
\item If\, $a_{n+1} \geq 23$ but $n$ is not a two-bridges level, note that all intervals $\Delta_j$ with $j\in\{12,...,a_{n+1}-11\}$ belong to $\widehat{\mathcal{P}}_{n+1}$, since the replacement procedure described in the previous section will not affect their vertices. In this case, we consider a single bridge of\, $\widehat{\mathcal{P}}_{n+1}|_{I_n(c_0)}$, which is the interval $\widehat{G}_1$ given by$$\widehat{G}_1=\bigcup_{j=12}^{a_{n+1}-11}\Delta_j\,.$$
\end{itemize}
In both cases, the bridges of\, $\widehat{\mathcal{P}}_{n+1}$ are the iterates, between $0$ and $q_{n+1}-1$, of the bridges of $\widehat{\mathcal{P}}_{n+1}|_{I_n(c_0)}$. Finally, if $a_{n+1} \leq 22$, no bridges are defined for $\widehat{\mathcal{P}}_{n+1}$.
\end{definition}
   
   \begin{figure}[!ht]
  \centering
  \psfrag{In+2}[][][1]{$I_{n+2}$} 
  \psfrag{F}[][][1]{$\widehat{G}_2$}
  \psfrag{c}[][][1]{$\mathfrak{c}_{n}$}
  \psfrag{A}[][][1]{$\widehat{\Delta}_n$} 
  \psfrag{B}[][][1]{$\widehat{\Delta}_n^{-1}$}
  \psfrag{G}[][][1]{$\widehat{G}_1$}
  \psfrag{Delta}[][][1]{$\Delta_1$}
  \includegraphics[width=4.8in]{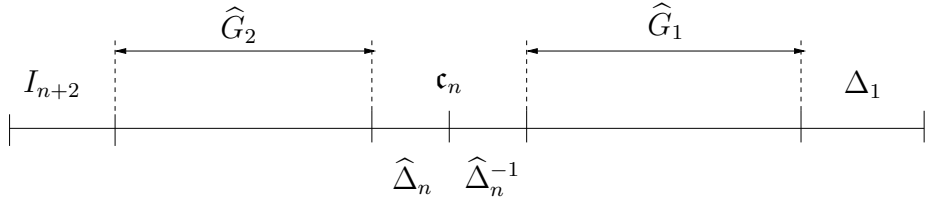}
  \caption{\label{figurahatG1hatG2} The bridges $\widehat{G}_1$ and $\widehat{G}_2$ contained in $I_n(c_0)$, for a two-bridges level $n$.}
  \end{figure}

Let $I$ be a bridge of $\widehat{\mathcal{P}}_{n+1}$, that is, $I$ is the union of a certain number of adjacent intervals belonging to $\widehat{\mathcal{P}}_{n+1}$. Following the terminology used in \cite{dFdM} and \cite{EdF}, we say that $I$ is a
\begin{enumerate}
\item[(a)] {\it Regular bridge\/}, if the bridge is formed by less than $1000$ intervals of $\widehat{\mathcal{P}}_{n+1}$.
\item[(b)] {\it Saddle-node bridge\/}, if the bridge is formed by at least $1000$ intervals of $\widehat{\mathcal{P}}_{n+1}$.
\end{enumerate}Any atom of $\widehat{\mathcal{P}}_{n+1}$ disjoint from all bridges will be called a \emph{regular interval}. In particular, if\, $a_{n+1} \leq 22$, all intervals of $\widehat{\mathcal{P}}_{n+1}$ are regular, since no bridges were defined.

Finally, the \emph{intermediate} partition $\widetilde{\mathcal{P}}_{n+1}$ is defined as the union of all regular intervals and all bridges (regular or saddle-node) of the auxiliary partition $\widehat{\mathcal{P}}_{n+1}$ (note that $\widetilde{\mathcal{P}}_{n+1}$ is finer than $\widehat{\mathcal{P}}_n$ but coarser than $\widehat{\mathcal{P}}_{n+1}$, which is why we call it intermediate).

\begin{remark}\label{max24} Any atom of $\widehat{\mathcal{P}}_n$ is the union of at most $48$ atoms of $\widetilde{\mathcal{P}}_{n+1}$. This fact will be mentioned in Section \ref{secfinalFG} below, during the proof of Proposition \ref{gridpprop}.
\end{remark}

The following lemma shows that all intervals of $\widetilde{\mathcal{P}}_{n+1}$ contained in the same atom of $\widehat{\mathcal{P}}_n$ are pairwise comparable.

\begin{lemma}\label{regularandbridges} Any interval of the intermediate partition $\widetilde{\mathcal{P}}_{n+1}$ is comparable to the interval of the auxiliary partition $\widehat{\mathcal{P}}_n$ that contains it.
\end{lemma}

In the proof of Lemma \ref{regularandbridges} we will use the following fact, which is \cite[Lemma 4.2]{dFG20}.

\begin{lemma}\label{lemmaiteratescritspots} Let $n$ be a two-bridges level, and let $j\in\{1,...,a_{n+1}\}$ be such that the interval $\Delta_j=f^{(j-1)q_{n+1}+q_n}\big(I_{n+1}(c_0)\big)\subset I_n(c_0)$ contains the free critical point $\mathfrak{c}_n$ of $f^{q_{n+1}}$. Then$$\left|f^i\big(\Delta_j\big)\right|\asymp \left|f^i\big(I_n(c_0)\big)\right|\quad\mbox{for all $i\in\{0,...,q_{n+1}-1\}$.}$$
\end{lemma}

\begin{proof}[Proof of Lemma \ref{regularandbridges}] Note first that, by Lemma \ref{lemaintersect}, it is enough to prove that any regular interval or bridge of $\widehat{\mathcal{P}}_{n+1}$ is comparable to the interval of $\mathcal{P}_n(c_0)$ that contains it.
\begin{itemize}
\item If $a_{n+1} \leq 22$, we have $\widetilde{\mathcal{P}}_{n+1}=\widehat{\mathcal{P}}_{n+1}$ and then Lemma \ref{regularandbridges} follows from the real bounds (Theorem \ref{teobeau}) and the fact, already mentioned, that the replacement procedure (to build $\widehat{\mathcal{P}}_{n+1}$ from $\mathcal{P}_{n+1}(c_0)$) creates no small atoms.
\item If $a_{n+1} \geq 23$ but $n$ is not a two-bridges level, we have two different types of elements in $\widetilde{\mathcal{P}}_{n+1}$.
\begin{itemize}
\item By the real bounds, any interval of $\mathcal{P}_{n+1}(c_0)$ of the form $f^i\big(I_{n+1}(c_0)\big)$ with $0\leq i\leq q_{n}-1$, $f^i\big(I_{n+2}(c_0)\big)$ with $0\leq i\leq q_{n+1}-1$, or $f^i(\Delta_j)$ with $j\in\{1,...,11\}\cup\{a_{n+1}-10,...,a_{n+1}\}$ and $0\leq i\leq q_{n+1}-1$ is comparable to the interval of $\mathcal{P}_n(c_0)$ that contains it. Using again that the replacement procedure creates no small atoms, we deduce Lemma \ref{regularandbridges} for any regular interval of $\widehat{\mathcal{P}}_{n+1}$ which is not a bridge.
\item Any bridge of $\widehat{\mathcal{P}}_{n+1}$ contains an interval which is adjacent to one of the intervals considered in the previous item, and we are done by the real bounds and Property \eqref{itemboundgeom} of the auxiliary partitions.
\end{itemize}
\item For a two-bridges level $n$, we have four different types of elements in $\widetilde{\mathcal{P}}_{n+1}$.
\begin{itemize}
\item The case of $I$ being an iterate of $I_{n+1}(c_0)$, $I_{n+2}(c_0)$ or $\Delta_1$ follows from the real bounds.
\item Let $I=f^{i}(\widehat{\Delta}_n)$ for some $i\in\{0, ... ,q_{n+1}-1\}$, and let $\Delta \in\mathcal{P}_{n+1}(c_0)$ be the interval that contains the free critical point $\mathfrak{c}_{n}$. By Lemma \ref{lemmaiteratescritspots} we have$$\left|f^i\big(\Delta\big)\right|\asymp \left|f^i\big(I_n(c_0)\big)\right|\quad\mbox{for all $i\in\{0,1,...,q_{n+1}-1\}$,}$$while by Corollary \ref{symmetricintervals} we have $|f^{i}(\widehat{\Delta}_n)|\asymp |f^{i}(\widehat{\Delta}_n^{-1})|$. Therefore, $|f^{i}(I_n(c_0))|\asymp |f^{i}(\Delta)|\leq |f^{i}(\widehat{\Delta}_n) \cup f^{i}(\widehat{\Delta}_n^{-1})|\asymp|f^{i}(\widehat{\Delta}_n)|$, see Figure \ref{DeltahatDelta}.
\item Just as in the previous item, we have $|f^{i}(\widehat{\Delta}_n^{-1})|\asymp|f^{i}(I_n(c_0))|$ for all $i\in\{0, ... ,q_{n+1}-1\}$.
\item Just as before, if $I$ is a bridge of $\widehat{\mathcal{P}}_{n+1}$, it contains an interval which is adjacent to one of the intervals in the previous items (and then we are done by the real bounds and Property \eqref{itemboundgeom} of the auxiliary partitions).
\end{itemize}
\end{itemize}
\end{proof}

\begin{figure}[!h]
\centering
\psfrag{F}[][]{$f^{q_{n+1}}(\Delta)$} 
\psfrag{D}[][][1]{$\Delta$}
\psfrag{G}[][][1]{$f^{-q_{n+1}}(\Delta)$} 
\psfrag{c}[][][1]{$\mathfrak{c}_{n}$}
\psfrag{A}[][][1]{$\widehat{\Delta}_n$}
\psfrag{B}[][][1]{$\widehat{\Delta}_n^{-1}$}
\includegraphics[width=2.6in]{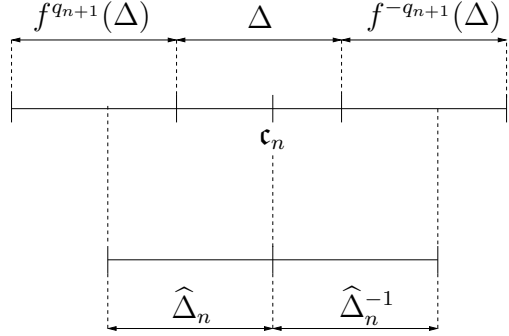}
\caption{\label{DeltahatDelta} The interval $\Delta$ of the standard partition $\mathcal{P}_{n+1}(c_0)$ containing the free critical point $\mathfrak{c}_{n}$, for a two-bridges level $n$.}
\end{figure}

\begin{remark}\label{obspontekoebe} For any given bridge $\widehat{G}_i\subset I_n(c_0)$ denote by $\widehat{G}_i^*\subset I_n(c_0)$ the union of $\widehat{G}_i$ with its two neighbours in the auxiliary partition $\widehat{\mathcal{P}}_{n+1}$, and note that the map $f^{q_{n+1}}: \mathrm{int}(\widehat{G}_i^*) \to f^{q_{n+1}}(\mathrm{int}(\widehat{G}_i^*))$ is a diffeomorphism. By Lemma \ref{regularandbridges}, both neighbours of $\widehat{G}_i$ are comparable to it (since they are comparable to $I_n(c_0)$), and the same happens to all its images up to time $q_{n+1}$. Therefore, by the standard Koebe's distortion principle \cite[Section IV.3, Theorem 3.1]{dMvS}, there exists a constant $K>1$, depending only on the real bounds, such that$$\frac{1}{K} \leq \frac{Df^j(x)}{Df^j(y)} \leq K \quad\mbox{for all $x,y\in\widehat{G}_i$ and $j\in\{1,...,q_{n+1}\}$.}$$This remark will be useful in Section \ref{secbalanponte} below, in order to propagate some geometric bounds from bridges contained in $I_n(c_0)$ to any bridge along the unit circle.
\end{remark}

\subsubsection{Balanced decompositions of almost parabolic maps}\label{subsecalm} Recall that the \textit{Schwarzian derivative} of a one-dimensional $C^3$ map $\phi$ is the differential operator defined at regular points by:
 \begin{equation*}
  S\phi(x)= \dfrac{D^{3}\phi(x)}{D\phi(x)} - \dfrac{3}{2} \left(\dfrac{D^{2}\phi(x)}{D\phi(x)}\right)^{2}.
 \end{equation*}

\begin{definition} An \emph{almost parabolic map} is a $C^3$ diffeomorphism 
\[
 \phi:\, J_1\cup J_2\cup \cdots \cup J_\ell \;\to\; J_2\cup J_3\cup \cdots \cup J_{\ell+1} \ ,
\]
where $J_1,J_2, \ldots, J_{\ell+1}$ are adjacent intervals on the circle (or on the line), with the following properties. 
\begin{enumerate}
 \item[(i)] One has $\phi(J_\nu)= J_{\nu+1}$ for all $1\leq \nu\leq \ell$;
 \item[(ii)] The Schwarzian derivative of $\phi$ is everywhere negative.
\end{enumerate}
The positive integer $\ell$ is called the \emph{length} of $\phi$, and the positive real number 
\[
 \sigma =\min\left\{\frac{|J_1|}{|\cup_{\nu=1}^\ell J_\nu|}\,,\, \frac{|J_\ell|}{|\cup_{\nu=1}^\ell J_\nu|}     \right\}
\]
is called the \emph{width\/} of $\phi$.
\end{definition}

The following property about the geometry 
of the fundamental domains of an almost parabolic map is due to J.-C. Yoccoz.

\begin{lemma}[Yoccoz]\label{lemyoccoz}
 Let $\phi: \bigcup_{\nu=1}^\ell J_\nu \to \bigcup_{\nu=2}^{\ell+1} J_\nu$ be an almost parabolic map with length $\ell$ and 
width $\sigma$. There exists a constant $C_\sigma>1$ (depending on $\sigma$ but not on $\ell$) such that, 
for all $\nu=1,2,\ldots,\ell$, we have
\begin{equation}\label{yocineq}
 \frac{C_\sigma^{-1}|I|}{[\min\{\nu,\ell+1-\nu\}]^2} \;\leq\; |J_\nu| \;\leq\;  \frac{C_\sigma|I|}{[\min\{\nu,\ell+1-\nu\}]^2}\ ,
\end{equation}
where $I=\bigcup_{\nu=1}^\ell J_\nu$ is the domain of $\phi$. 
\end{lemma}

A proof of this lemma can be found in \cite[Appendix B, page 386]{dFdM}. The following result is \cite[Lemma 4.5]{EdF}, and its proof is a fairly immediate application of Lemma \ref{lemyoccoz}.

\begin{lemma}\label{balancedecomp} Let $\phi$ be an almost parabolic map with domain $I=\bigcup_{\nu=1}^\ell J_\nu$, and let $d\in \mathbb{N}$ be largest 
such that $2^{d+1}\leq \ell/2$. There exists a descending chain of (closed) intervals
\[
 I=M_0\supset M_1\supset \cdots \supset M_{d+1}
\]
for which, letting $L_i,R_i$ denote the (left and right) connected components of $M_i\setminus M_{i+1}$ for all 
$0\leq i\leq d$, the following properties hold.
\begin{enumerate}
 \item[(i)] Each of the intervals $L_i,R_i$ is the union of exactly $2^i$ adjacent atoms (fundamental domains) 
of $I$.
 \item[(ii)] We have
\begin{equation}\label{balance1}
 I\;=\; \bigcup_{i=0}^{d}L_i\;\cup\;M_{d+1}\;\cup\; \bigcup_{i=0}^{d}R_i \ . 
\end{equation}
 \item[(iii)] For each $0\leq i\leq d$ we have $|L_i|\asymp |M_{i+1}|\asymp |R_i|$, with comparability constants 
depending only on the width $\sigma$ of $\phi$. 
\end{enumerate}
\end{lemma}

A decomposition of the form \eqref{balance1} satisfying properties (i), (ii), (iii) of Lemma \ref{balancedecomp} is called a \emph{balanced decomposition} of $I$. The intervals 
$M_i$, $0\leq i\leq d+1$, are said to be \emph{central}, whereas the intervals $L_i,R_i$, $0\leq i\leq d$, are 
said to be \emph{lateral}.

\begin{remark}\label{ratioofJs} As it follows from Yoccoz's Lemma \ref{lemyoccoz}, the following fact holds true for the fundamental domains $J_\nu$ ($1\leq \nu\leq \ell$) of any almost parabolic map $\phi$\,: for all $1\leq k<l<m\leq \ell$, one has
\[
 \frac{|J_{l+1}|+|J_{l+2}|+\cdots+|J_m|}{|J_{k+1}|+|J_{k+2}|+\cdots+|J_l|}\;\asymp\; \frac{k(m-l)}{m(l-k)}\ ,
\]with comparability constant depending only on the width $\sigma$ of $\phi$. In particular, if the interval $\bigcup_{k+1}^{m}J_{\nu}$ is contained in a lateral or the \emph{final} central interval of a balanced decomposition as in \eqref{balance1}, and if $m-l$ is at most four times larger than $l-k$ (and vice versa), one has $|J_{k+1}|+|J_{k+2}|+\cdots+|J_l| \asymp |J_{l+1}|+|J_{l+2}|+\cdots+|J_m|$, again with comparability constant depending only on the width of $\phi$. This fact will be useful in Section \ref{secfinalFG}, during the proof of Proposition \ref{gridpprop}.
\end{remark}

\subsubsection{Balanced decompositions of bridges}\label{secbalanponte} We recall now \cite[Lemma 4.1]{EdFG}.

\begin{lemma}\label{negschwarz} For any given multicritical circle map $f$ and any critical point $c_0 \in S^1$ of $f$ there exists $n_0=n_0(f)\in\nt$ such that for all $n \geq n_0$ we have
that
\[ Sf^{j}(x)<0\quad\text{for all $j\in \{1, \cdots, q_{n+1}\}$ and for all $x \in I_{n}(c_0)$ regular point of $f^{j}$.}
 \]
Likewise, we have
\[ Sf^{j}(x)<0\quad\text{for all $j\in \{1, \cdots, q_{n}\}$ and for all $x \in I_{n+1}(c_0)$ regular point of $f^j$}.
 \]
\end{lemma}

At this point, we would like to apply Lemma \ref{balancedecomp} to any saddle-node bridge $\widehat{G}_i$ of $\widehat{\mathcal{P}}_{n+1}|_{I_n(c_0)}$, for $n \geq n_0$. Indeed, by construction, the map $\phi=f^{q_{n+1}}|_{\widehat{G}_i}$ has no critical points, hence it is a diffeomorphism onto its image. Moreover, by Lemma \ref{negschwarz}, $\phi$ has negative Schwarzian derivative. Finally, note that by Lemma \ref{regularandbridges} and Property \eqref{itemboundgeom} of the auxiliary partitions, the width of $\phi$ only depends on the real bounds. The only problem seems to be that, for a two-bridges level $n$, both bridges $\widehat{G}_1$ and $\widehat{G}_2$ of $\widehat{\mathcal{P}}_{n+1}|_{I_n(c_0)}$ contain an element of the auxiliary partition $\widehat{\mathcal{P}}_{n+1}$ which may not be a fundamental domain for $f^{q_{n+1}}$, namely $\Delta_{R}^{n+1}$ and $\Delta_{L}^{n+1}$ respectively. However, both of these intervals contain a fundamental domain and are contained in the union of such fundamental domain with one of its adjacent fundamental domains. Therefore, estimate \eqref{yocineq} of Lemma \ref{lemyoccoz} still holds, just by adjusting constants.

In other words, there exists a balanced decomposition for any saddle-node bridge of $\widehat{\mathcal{P}}_{n+1}|_{I_n(c_0)}$ (with uniform comparability constants, depending only on the real bounds for $f$). With this at hand, we use Remark \ref{obspontekoebe} to spread this decomposition to all bridges around the unit circle, in order to obtain the following result, which is the goal of both sections \ref{subsecalm} and \ref{secbalanponte}.

\begin{lemma}\label{generalbalancedecomp} There exists a balanced decomposition for any saddle-node bridge of $\widehat{\mathcal{P}}_{n+1}$, with uniform comparability constants depending only on the real bounds for $f$.
\end{lemma}

We finish Section \ref{secbalanponte} with the following estimate (borrowed from \cite[Lemma 4.11]{dFdM}), that will be useful in Section \ref{SecProofThmB} (during the proof of Lemma \ref{ineqbeta}).
  
\begin{lemma}\label{corolariolemayoccoz} There exists a constant $M>1$ depending only on the real bounds such that for all $k,p\in\nt$ we have that if $I\in\widehat{\mathcal{P}}_{k}$, $J\in\widehat{\mathcal{P}}_{k+p}$ and $J \subset I$, then$$|I| \leq M^p \,(a_{k+1}\,a_{k+2}\,\dots\,a_{k+p})^2\,|J|\,.$$Moreover, the same estimate holds replacing $J$ with its return $f^{q_{k+1}}(J)$.
\end{lemma}

Lemma \ref{corolariolemayoccoz} follows from Lemma \ref{lematrayectoria}, Lemma \ref{regularandbridges}, Lemma \ref{negschwarz}, Yoccoz's Lemma \ref{lemyoccoz}, Remark \ref{obspontekoebe} and a simple inductive argument.

\subsubsection{The fine grid}\label{secfinalFG} In Proposition \ref{gridpprop} below we finally construct a fine grid for $f$. The partition $\mathcal{Q}_n$ that we want is constructed from $\widehat{\mathcal{P}}_m$ and $\widetilde{\mathcal{P}}_m$ for various values of $m\leq n$. At this point, our construction  is essentially the same as in \cite[pages 359-361]{dFdM} or \cite[pages 5612-5614]{EdF}. We reproduce it here just for the convenience of the reader.

\begin{proposition}\label{gridpprop} There exists a fine grid $\{\mathcal{Q}_n\}$ in $S^1$ with the following properties.
\begin{enumerate}
\item[($a$)] Every atom of $\mathcal{Q}_n$ is the union of at most $b=1000$ atoms of $\mathcal{Q}_{n+1}$.
\item[($b$)] Every atom $\Delta\in \mathcal{Q}_n$ is a union of atoms of $\widehat{\mathcal{P}}_m$ and $\widetilde{\mathcal{P}}_m$ for some $m\leq n$, and there are four possibilities:
\begin{enumerate}
\item[($b_1$)] $\Delta$ is a single atom of $\widehat{\mathcal{P}}_m$, contained in a bridge atom of $\widetilde{\mathcal{P}}_m$;
\item[($b_2$)] $\Delta$ is a single atom of $\widetilde{\mathcal{P}}_m$;
\item[($b_3$)] $\Delta$ is a central interval of a saddle-node bridge atom of $\widetilde{\mathcal{P}}_m$;
\item[($b_4$)] $\Delta$ is the union of at least two atoms of $\widehat{\mathcal{P}}_{m}$, contained in a saddle-node bridge atom of $\widetilde{\mathcal{P}}_m$.
\end{enumerate}
\end{enumerate}
\end{proposition}

\begin{remark}\label{remvertFG} Any vertex of $\mathcal{Q}_n$ is a vertex of $\widehat{\mathcal{P}}_m$ for some $m \leq n$, and then it is also a vertex of $\widehat{\mathcal{P}}_n$. In other words, the auxiliary partition $\widehat{\mathcal{P}}_n$ is a refinement of the fine grid~$\mathcal{Q}_n$.
\end{remark}

\begin{proof}[Proof of Proposition \ref{gridpprop}] The proof is by induction on $n$. The first partition $\mathcal{Q}_1$ consists of all atoms of $\widetilde{\mathcal{P}}_1$ which are not saddle-node atoms, together with the intervals $L_0$, $M_1$ and $R_0$ of each saddle-node interval $I\in\widetilde{\mathcal{P}}_1$ ($I=L_0\cup M_1\cup R_0$). It is clear that each atom of $\mathcal{Q}_1$ falls within one of the categories ($b_1$)-($b_3$) above. 

Assuming $\mathcal{Q}_n$ has been defined, we define $\mathcal{Q}_{n+1}$ as follows. Take an atom $I\in \mathcal{Q}_n$ and consider the four cases below.
\begin{enumerate}
\item[(1)] If $I$ is a single atom of $\widehat{\mathcal{P}}_{m}$, we break it into the union of at most $48$ atoms of $\widetilde{\mathcal{P}}_{m+1}$ (recall Remark \ref{max24}), and take them as atoms of $\mathcal{Q}_{n+1}$, all of which are of type ($b_2$).
\item[(2)] If $I$ is a single atom of $\widetilde{\mathcal{P}}_m$, then one of two things can happen:
\begin{enumerate}
\item[(i)] $I$ is a saddle-node atom: In this case write $I=L_0\cup M_1\cup R_0$ as above and take $L_0$, $R_0$ and $M_1$ as atoms of $\mathcal{Q}_{n+1}$. Note that the lateral intervals $L_0$ and $R_0$ are atoms of type ($b_1$), while the central interval $M_1$ is of type ($b_3$).
\item[(ii)] $I$ is not a saddle-node atom: Here, there are two sub-cases to consider. The first possibility is that $I$ is a single (regular) atom of $\widehat{\mathcal{P}}_{m}$, and we proceed as in Item $(1)$ above. The second possibility is that $I$ is a (regular) bridge, in which case we break it up into its $\leq 1000$ constituent atoms of $\widehat{\mathcal{P}}_{m+1}$, and take them as atoms of $\mathcal{Q}_{n+1}$, all of which are of type ($b_1$). 
\end{enumerate}
\item[(3)] If $I$ is a central interval of a saddle-node bridge atom of $\widetilde{\mathcal{P}}_m$, then one of two things can happen. If $I$ is the \emph{final} central interval, proceed as in Item $(4)$ below (unless $I$ is just a single atom of $\widehat{\mathcal{P}}_{m}$, in which case and we proceed as in Item $(1)$ above). If $I$ is a central interval which is not the final interval, consider the next central interval inside $I$, say $M$, and the two corresponding lateral intervals $L$ and $R$ such that $I=L\cup M\cup R$, and declare $L$, $R$ and $M$ members of $\mathcal{Q}_{n+1}$. Note that $L$ and $R$ are of type ($b_4$), while $M$ is of type ($b_3$).
\item[(4)] If $I$ is a union of $p\ge 2$ consecutive atoms
$J_1,\ldots,J_p$ of $\widehat{\mathcal{P}}_{m}$ inside a saddle-node bridge atom of $\widetilde{\mathcal{P}}_m$, divide it up into two approximately equal parts. More precisely, write $p=2q+r$, where $r=0$ or $1$, and consider $I=L\cup R$
where 
$$
L\;=\;\bigcup_{j=1}^qJ_j \ ,
\ R\;=\;\bigcup_{j=q+1}^pJ_j
\ .
$$
We obtain in this fashion two new atoms of $\mathcal{Q}_{n+1}$ (namely $L$ and $R$) which are either single atoms of 
$\widehat{\mathcal{P}}_{m}$, and therefore of type ($b_1$), or once again intervals of type ($b_4$).
\end{enumerate}
This completes the induction. Both Item \eqref{item1FG} and Item \eqref{item2FG} in Definition \ref{def:finegrid} follow directly from our construction, so we finish the proof of Proposition \ref{gridpprop} verifying Item \eqref{item3FG}. Given two adjacent atoms $\Delta,\Delta'\in \mathcal{Q}_n$, there are three cases to consider.
\begin{enumerate}
\item[(a)] There exist $m, m'\leq n$ such that $\Delta$ is a single atom of $\widehat{\mathcal{P}}_m$ and $\Delta'$ is a single atom of $\widehat{\mathcal{P}}_{m'}$. In this case, either $m=m'$, or $m$ and $m'$ differ by $1$ (this is easily proved by induction on $n$ from the construction of $\mathcal{Q}_n$ given above). But then we have $|\Delta|\asymp |\Delta'|$ by Property \eqref{itemboundgeom} of the auxiliary partitions (see Section \ref{secinter}).
\item[(b)] There exist $m, m'\leq n$ such that $\Delta$ is a single atom of $\widetilde{\mathcal{P}}_m$ and $\Delta'$ is a single atom of $\widetilde{\mathcal{P}}_{m'}$. This case is analogous to the previous one, just replacing Property \eqref{itemboundgeom} with Lemma \ref{regularandbridges}.
\item[(c)] For some $m\leq n$, at least one of the two atoms, say $\Delta$, is the union of $p\geq 2$ atoms of $\widehat{\mathcal{P}}_{m}$ inside a single atom of $\widetilde{\mathcal{P}}_m$, which is necessarily a bridge $\widehat{G}\in\widetilde{\mathcal{P}}_m$. If we are not in the previous cases, then $\Delta'$ is also contained in the bridge $\widehat{G}$. Looking at the balanced decomposition of $\widehat{G}$ (given by Lemma \ref{generalbalancedecomp}), we see that there are four possibilities.
\begin{itemize}
\item The first possibility is that $\Delta$ is a lateral interval ($L_i$ or $R_i$) and $\Delta'$ is the corresponding central interval $M_{i+1}$ of the balanced decomposition of $\widehat{G}$, for some $i\in\{0,..,d\}$. This case follows from Property (iii) of Lemma \ref{balancedecomp}.
\item The second possibility is that $\Delta$ is contained in a final lateral interval ($L_d$ or $R_d$) and $\Delta'$ is contained in the final central interval $M_{d+1}$ of the balanced decomposition of $\widehat{G}$. This case follows from Property (iii) of Lemma \ref{balancedecomp} and Remark \ref{ratioofJs}.
\item The third possibility is that both $\Delta$ and $\Delta'$ are contained in the same lateral interval ($L_i,R_i$) or the same final central interval ($M_{d+1}$) of said balanced decomposition. In this case, the number of fundamental domains of $\widehat{G}$ inside $\Delta$ differs at most by $1$ from the number of those inside $\Delta'$, and then we have $|\Delta|\asymp |\Delta'|$ by Remark \ref{ratioofJs}.
\item The fourth possibility is that $\Delta$ and $\Delta'$ are contained in adjacent intervals of the balanced decomposition of $\widehat{G}$, both being lateral intervals. In this case, one of the two atoms, $\Delta$ or $\Delta'$, is the union of at most four times fundamental domains of $\widehat{G}$ than the other, and we have $|\Delta|\asymp |\Delta'|$, again by Remark \ref{ratioofJs}.
\end{itemize}
\end{enumerate}
This establishes the desired comparability of adjacent atoms of $\mathcal{Q}_n$ in all cases, with uniform constants depending only on the real bounds, and the proof of Proposition \ref{gridpprop} is complete.
\end{proof}

\section{Proof of Theorem \ref{theoremb}}\label{SecProofThmB}

This final section is devoted to the proof of Theorem \ref{theoremb} (stated in Section \ref{secreduction}). With this purpose, let $f$ and $g$ be $C^3$ bicritical circle maps with the same irrational rotation number, which is contained in the set $\mathcal{A}$ (Definition \ref{setA}). Assume that $f$ and $g$ have the same signature (Definition \ref{signature}) and that the two critical points of $f$, say $c_0$ and $c_1$, are also the critical points of $g$. Assume, finally, that there exist $C>1$ and $0<\mu<1$ such that
\begin{equation}\label{hypTeoB}
\left|\,\frac{|I_{n}^g(c_i)|}{|I_{n}^f(c_i)|}- 1\,\right| \leq C\, \mu^{n} \quad\mbox{and}\quad d_{1}(\mathcal{R}_i^nf, \mathcal{R}_i^ng) \leq C\, \mu^{n}\,,
\end{equation}both for $i=0$ and $i=1$, and for all $n\in\nt$. Using Proposition \ref{criterion} and the fine grid constructed in Section \ref{sec:anewpartition}, we will prove that the topological conjugacy $h$ between $f$ and $g$ that fixes both $c_0$ and $c_1$ is a $C^{1+\alpha}$ diffeomorphism. Our first goal is the following result.

\begin{proposition}[Key estimate]\label{corokey} There exist constants $C_1>0$ and $\mu_1\in(0,1)$ depending only on the real bounds for $f$ such that$$\|h-\Id\|_{C^0(J_n^{f}(c_0))}\leq C_1\,\big|J_n^{f}(c_0)\big|\,\mu_{1}^{n}\quad\mbox{for all $n\in\nt$.}$$
\end{proposition}

Proposition \ref{corokey} will be a consequence of the following lemma, borrowed from \cite[Lemma 4.10]{dFdM}.

\begin{lemma}\label{keylemma1} There exist constants $C>0$, $K>1$ and $\mu_{*} \in (0,1)$ for which the following holds. Let $n,p \in \nt$ and let $v$ be a vertex of the auxiliary partition $\widehat{\mathcal{P}}_{n+p}$ (of $f$ around $c_0$) contained in $J_{n}^{f}(c_0)$. Then$$\big|v-h(v)\big| \leq C\,K^{p}\,\big|J_n^{f}(c_0)\big|\,\mu_{*}^{n}\,.$$Moreover, the three constants $C$, $K$ and $\mu_{*}$ only depend on the real bounds.
\end{lemma}

Before giving the proof of Lemma \ref{keylemma1}, let us see why it implies Proposition \ref{corokey}.

\begin{proof}[Proof of Proposition \ref{corokey}] Let $\mu\in(0,1)$ be given by Corollary \ref{mu3} and let $C>0$, $K>1$ and $\mu_{*} \in (0,1)$ be given by Lemma \ref{keylemma1}. We define $C_1=C_1(\mu,C)$ and $\mu_1=\mu_1(\mu,K,\mu_*)$ as follows: fix some $\sigma\in(0,1)$ with $\sigma<-\log\mu_*/\log K$, and let $\mu_1=\max\{K^{\sigma}\mu_*\,,\,\mu^{\sigma}\}$ and $C_1=C+1/\mu$. Now given $n\in\nt$ let $p=\lfloor \sigma n \rfloor$, and given $x \in J_n^{f}(c_0)$ let $v_1$ and $v_2$ be the endpoints of the atom of the auxiliary partition $\widehat{\mathcal{P}}_{n+p}$ that contains $x$. Since both $x$ and $h(x)$ belong to the convex hull of\, $[v_1,v_2]\cup\big[h(v_1),h(v_2)\big]$, we deduce from Corollary \ref{mu3} and Lemma \ref{keylemma1} that
\begin{align*}
\big|x-h(x)\big|&\leq\big|v_1-v_2\big|+\max_{i\in\{1,2\}}\big|v_i-h(v_i)\big|\leq\mu^p\,\big|J_n^{f}(c_0)\big|+C\,K^{p}\,\big|J_n^{f}(c_0)\big|\,\mu_{*}^{n}\\
&\leq\mu^{-1}(\mu^{\sigma})^n\,\big|J_n^{f}(c_0)\big|+C\,(K^{\sigma}\mu_*)^n\,\big|J_n^{f}(c_0)\big|\leq C_1\,\big|J_n^{f}(c_0)\big|\,\mu_{1}^{n}\,.
\end{align*}
\end{proof}

\begin{remark}\label{remkeyprop} Recall that the auxiliary partitions $\big\{\widehat{\mathcal{P}}_n\big\}_{n\in\nt}$ built in Section \ref{secinter} around $c_0$, can also be constructed around $c_1$. Since we are assuming exponential convergence of renormalization for both critical points (as in \eqref{hypTeoB} above), the statement of Lemma \ref{keylemma1} remains valid for those partitions around $c_1$. In particular, Proposition \ref{corokey} holds in $J_n^{f}(c_1)$ as well. This will be useful later, during the proof of Lemma~\ref{ineqbeta}.
\end{remark}

We proceed to make some comments before entering the proof of Lemma \ref{keylemma1}. First of all, since we are renormalizing around $c_0$, let us write $I_{n}^{f}$, $I_{n+1}^{f}$ and $J_{n}^{f}$ instead of $I_{n}^{f}(c_0)$, $I_{n+1}^{f}(c_0)$ and $J_{n}^{f}(c_0)$ respectively (and the same for $g$). For each $n\in\nt$ consider the \emph{$n$-th scaling ratio} of $f$, which is the positive number $s_n^f$ defined as$$s_n^f=\frac{\big|I_{n+1}^{f}\big|}{\big|I_{n}^{f}\big|}\,.$$By the real bounds (Theorem \ref{teobeau}), the sequence $\big\{s_n^f\big\}_{n\in\nt}$ is bounded away from zero and infinity. Recall from Section \ref{rmcm} that the $n$-th renormalization of $f$ at $c_0$ is the normalized multicritical commuting pair $\mathcal{R}^nf:[-s_{n}^{f},1]\to[-s_{n}^{f},1]$ given by
\[\mathcal{R}^nf=
\begin{dcases}
B_{n,f} \circ f^{q_{n}} \circ B_{n,f}^{-1} & \mbox{in $[-s_{n}^{f},0]$}\\[1ex]
B_{n,f} \circ f^{q_{n+1}} \circ B_{n,f}^{-1} & \mbox{in $[0,1]$\,,}\\
\end{dcases}
\]where $B_{n,f}$ is the unique orientation-preserving affine diffeomorphism between $J_{n}^{f}$ and $[-s_n^f,1]$.

\begin{remark}\label{remmeteq} For any $n\in\nt$ we have
\begin{equation}\label{eqdone}
\big\|\mathcal{R}^{n}f-\mathcal{R}^{n}g\big\|_{C^1([-\min\{s_n^f,s_n^g\},1])} \leq C\,\mu^{n}\,.
\end{equation}
Indeed, by hypothesis (recall Definition \ref{pseudometric}), the difference $\big|s_n^f-s_n^g\big|$ goes to zero exponentially fast. Since both sequences $\big\{s_n^f\big\}_{n\in\nt}$ and $\big\{s_n^g\big\}_{n\in\nt}$ are bounded away from zero and infinity, the two M\"obius transformations fixing $0$ and $1$, and mapping $-s_n^f$ ($-s_n^g$ respectively) to $-1$ converge together exponentially fast (and also its inverses). Since $d_{1}(\mathcal{R}^nf, \mathcal{R}^ng) \leq C\, \mu^{n}$, we deduce \eqref{eqdone}.
\end{remark}

For all $i \geq n$ let $A_{i,f}=B_{n,f} \circ B_{i,f}^{-1}$\,, which is just the linear contraction given by$$A_{i,f}(t)=\frac{\big|I_i^f\big|}{\big|I_n^f\big|}\,t\,.$$Consider the interval $\Lambda_i^f\subset[-s_n^f,1]$ given by $\Lambda_i^f=B_{n,f}(J_{i}^f)$, and consider the bicritical commuting pair $f_{i}^{*}\colon\Lambda_i^f\to\Lambda_i^f$ given by\[f_{i}^{*}=A_{i,f} \circ \mathcal{R}^{i}f\circ A_{i,f}^{-1}\,.\]The contraction $A_{i,f}$, the interval $\Lambda_i^f$ and the pair $f_{i}^{*}$ depend on $n$. However, since $n$ is fixed, we omit to mention it just to simplify notation. In the same way, define the corresponding objects for $g$.

\begin{lemma}\label{ineqf*} We have$$\big\|f_{i}^{*}-g_{i}^{*}\big\|_{C^0(\Lambda_i^f\cap\Lambda_i^g)} \leq C\, \mu^{n}\,\frac{\big|I_{i}^f\big|}{\big|I_{n}^f\big|}\,.$$\end{lemma}

\begin{proof}[Proof of Lemma \ref{ineqf*}] Take some $y\in\Lambda_i^f\cap\Lambda_i^g$, and let $y_f=A_{i,f}^{-1}(y)$ and $y_g=A_{i,g}^{-1}(y)$. Note that $|y|\leq\big|J_i^f\big|/\big|I_{n}^f\big|$. By combining the real bounds (Theorem \ref{teobeau}) with estimate \eqref{estratios} in Lemma \ref{corolemma4.7} we obtain$$|y_f-y_g|=\left|\frac{|I_{n}^f|}{|I_{i}^f|}- \frac{|I_{n}^g|}{|I_{i}^g|}\right|\,|y|\leq C_1\,\mu_1^{n}\,\frac{|I_{n}^f|}{|I_{i}^f|}\,\frac{|J_{i}^f|}{|I_{n}^f|}=C_1\,\mu_1^{n}\,\frac{|J_{i}^f|}{|I_{i}^f|}\leq C_2\,\mu_1^{n}.$$From Corollary \ref{c1bounds} (the $C^1$-bounds) and the exponential convergence of renormalization (recall Remark \ref{remmeteq}), we deduce that$$\big|\mathcal{R}^{i}f(y_f)-\mathcal{R}^{i}g(y_g)\big| \leq C_3\,|y_f-y_g|+C_4\,\mu_2^{i} \leq C_3\,C_2\,\mu_1^n+C_4\,\mu_2^i \leq C_5\,\mu_3^{n}\,,$$where $\mu_3=\max\{\mu_1,\mu_2\}$. Using again Lemma \ref{corolemma4.7} we finally obtain
\begin{align*}
\big|f_{i}^{*}(y)-g_{i}^{*}(y)\big|&=\big|A_{i,f}\big(\mathcal{R}^{i}f(y_f)\big)-A_{i,g}\big(\mathcal{R}^{i}g(y_g)\big)\big|\\
&\leq\frac{\big|I_i^f\big|}{\big|I_n^f\big|}\,\big|\mathcal{R}^{i}f(y_f)-\mathcal{R}^{i}g(y_g)\big|+\left|\frac{|I_{i}^f|}{|I_{n}^f|}- \frac{|I_{i}^g|}{|I_{n}^g|}\right|\,\big|\mathcal{R}^{i}g(y_g)\big|\\
&\leq C_5\,\mu_3^{n}\,\frac{\big|I_i^f\big|}{\big|I_n^f\big|}+C_1\,\mu_1^{n}\,\frac{\big|I_i^f\big|}{\big|I_n^f\big|}\,\max\{1,s_i^g\}\leq C_6\,\mu_3^{n}\,\frac{\big|I_i^f\big|}{\big|I_n^f\big|}\,.
\end{align*}
\end{proof}

\begin{remark}\label{remdistcrit} As it is not difficult to prove, there exists a constant $C_0>1$, depending only on the real bounds for $f$, with the following property: let $m$ be a two-bridges level for $f$ at $c_0$ (see Definition~\ref{deftwobridgeslevel}), and let $\mathfrak{c}_m^f$ be the free critical point of the first return map of $f$ to $J_m^f$ (just as in Section \ref{secinter}). Note that $\mathfrak{c}_m^g=h(\mathfrak{c}_m^f)$ is the corresponding free critical point for the return of $g$ to $J_m^g$, and recall that $d>1$ is the maximum of the criticalities of $f$ at $c_0$ and $c_1$. Then$$\big|B_{m,f}(\mathfrak{c}_m^f)-B_{m,g}(\mathfrak{c}_m^g)\big| \leq C_0\,d_{1}(\mathcal{R}^mf,\mathcal{R}^mg)^{1/d}\,.$$
\end{remark}

Combined with hypothesis \eqref{hypTeoB}, Remark \ref{remdistcrit} gives us$$\big|B_{m,f}(\mathfrak{c}_m^f)-B_{m,g}(\mathfrak{c}_m^g)\big| \leq C_0\,C^{1/d}\,(\mu^m)^{1/d}=C_7\,(\mu^{1/d})^m\,.$$

This is the only place in this paper where we need the assumption that exponential contraction of renormalization  holds in the $C^1$ metric (instead of just $C^0$, which is the assumption in \cite{dFdM}), to be able to control the position of the critical points for the return maps. Finally, let us recall \cite[Proposition 4.1]{dFdM}.
  
\begin{proposition}\label{propoapm} Let $\phi$ and $\varphi$ be two almost parabolic maps with the same length $\ell$ defined on the same interval. Then for all $x \in J_1(\phi) \cap J_1(\varphi)$ and for all $0\leq k \leq \ell/2$, we have
\begin{equation*}
|\phi^k(x)-\varphi^k(x)| \leq C\, k^3\,\|\phi-\varphi\|_{C^0}.
\end{equation*}
\end{proposition}

Let us mention that Proposition \ref{propoapm}, which is based on the geometric inequalities given by Yoccoz's Lemma \ref{lemyoccoz}, has been significantly improved in \cite[Section 6]{GMdM} (see for instance Lemma 6.6 on page 2155 and Proposition 6.18 on page 2163). However, such sharper estimates will not be needed in the present paper. We are ready to start the proof of Lemma \ref{keylemma1}.
      
\begin{proof}[Proof of Lemma \ref{keylemma1}] By Lemma \ref{lematrayectoria} there exist $L\in\{1,...,p\}$ and $n \leq m_1 < ... < m_L \leq n+p$ such that\, $v = \varphi_1 \circ \cdots \circ \varphi_L(x)$\,, where:
\begin{itemize}
\item For each $j\in\{1,...,L\}$ we have $\varphi_j= f^{k_jq_{m_j+1}+\sigma_jq_{m_j}}$ for some $\sigma_j\in\{0,1\}$ and $k_j\in\Z$, where each $k_j$ either satisfies $|k_j|\leq\ell(m_j)$ and $|k_j| \leq r(m_j)$ or $|k_j| \leq \lceil a_{m_j+1}/2 \rceil$, depending on whether $m_j$ is or is not a two-bridges level for $f$ at $c_0$.
\item For each $j\in\{1, ... , L-1\}$, the point $\varphi_{j+1} \circ \cdots \circ \varphi_L(x)$ either belongs to $J_{m_j+1}^f\cup\widehat{\Delta}_{m_j}^f$ or to $J_{m_j+1}^f$, depending on whether $m_j$ is or is not a two-bridges level.
\item There exists $m\in\{m_L,...,n+p\}$ such that the initial condition $x$ either belongs to $\big\{c_0\,,\,f^{q_{m+2}}(c_0)\,,\,\mathfrak{c}_{m}^f\big\}$ or to $\big\{c_0\,,\,f^{q_{m+2}}(c_0)\big\}$, depending on whether $m$ is or is not a two-bridges level.
\end{itemize}

Let $w=h(v)$ and $y=h(x)$, and note that $w=\psi_1\circ\cdots\circ\psi_L(y)$, where $\psi_j=h \circ \varphi_j \circ h^{-1}$ for each $j\in\{1,...,L\}$. In order to estimate $|v-w|$ we will first estimate $|v^{*}-w^{*}|$\,, where $v^{*}=B_{n,f}(v)$ and $w^{*}=B_{n,g}(w)$. Let $\mu_2\in(0,1)$ be defined as $\mu_2=\max\{\mu^{1/d},\mu_1\}$, where $\mu$ is given by hypothesis \eqref{hypTeoB}, $d>1$ is the maximum of the criticalities of $f$ at $c_0$ and $c_1$, and $\mu_1$ is given by Corollary~\ref{corbeau}. We start by considering $x^{*}=B_{n,f}(x)$ and $y^{*}=B_{n,g}(y)$, and we claim that $|x^{*}-y^{*}| \leq C_8\,\mu_2^{m}$. Indeed, note that
\begin{align*}
x^{*}-y^{*}&=A_{m,f}\big(B_{m,f}(x)\big)-A_{m,f}\big(B_{m,g}(y)\big)+\big(A_{m,f}-A_{m,g}\big)\big(B_{m,g}(y)\big)\\
&=\frac{\big|I_{m}^f\big|}{\big|I_{n}^f\big|}\,\big(B_{m,f}(x)-B_{m,g}(y)\big)+\left(\frac{\big|I_{m}^f\big|}{\big|I_{n}^f\big|}-\frac{\big|I_{m}^g\big|}{\big|I_{n}^g\big|}\right)\,B_{m,g}(y).
\end{align*}From the exponential convergence of renormalization (recall remarks \ref{remmeteq} and \ref{remdistcrit} above) we know that $\big|B_{m,f}(x)-B_{m,g}(y)\big| \leq C_7\,(\mu^{1/d})^{m} \leq C_7\,\mu_2^m$, while from estimate \eqref{estratios} in Lemma \ref{corolemma4.7} we have$$\left|\frac{\big|I_{m}^f\big|}{\big|I_{n}^f\big|}-\frac{\big|I_{m}^g\big|}{\big|I_{n}^g\big|}\right| \leq C_1\,\mu^n\,\frac{\big|I_{m}^f\big|}{\big|I_{n}^f\big|}\leq C_1\,\mu_2^n\,\frac{\big|I_{m}^f\big|}{\big|I_{n}^f\big|}\,.$$From the real bounds (see Corollary \ref{corbeau}) we know that$$\big|I_{m}^f\big| \leq \mu_1^{m-n} \big|I_{n}^f\big|\leq \mu_2^{m-n} \big|I_{n}^f\big|\,,$$and since $\big|B_{m,g}(y)\big|\leq\max\{1,s_{m}^g\}$ is bounded, we obtain the claim.

Now for each $j\in\{1,...,L\}$ let $\varphi_{j}^{*}=B_{n,f}\circ \varphi_{j} \circ B_{n,f}^{-1}$ and $\psi_{j}^{*}=B_{n,g}\circ \psi_{j}\circ B_{n,g}^{-1}$. We claim that$$\big\|\varphi_{j}^{*}-\psi_{j}^{*}\big\|_{C^0(\Lambda_{m_j}^f\cap\Lambda_{m_j}^g)} \leq  C_9\,a_{m_j+1}^3\,\mu_2^{m_j}\,.$$Indeed, combining Proposition \ref{propoapm} with Lemma \ref{ineqf*} we have for any $j\in\{1,...,L\}$ and $y\in\Lambda_{m_j}^f\cap\Lambda_{m_j}^g$ that$$\big|\varphi_{j}^{*}(y)- \psi_{j}^{*}(y)\big| \leq  C_9\, |k_j|^3 \, \mu^{n}\,\frac{\big|I_{m_j}^f\big|}{\big|I_{n}^f\big|}\leq  C_9\, |k_j|^3 \, \mu_2^{n}\,\frac{\big|I_{m_j}^f\big|}{\big|I_{n}^f\big|}\,.$$As before, we know from the real bounds (see Corollary \ref{corbeau}) that$$\big|I_{m_j}^f\big| \leq \mu_1^{m_j-n} \big|I_{n}^f\big|\leq \mu_2^{m_j-n} \big|I_{n}^f\big|\,,$$and then$$\big|\varphi_{j}^{*}(y)- \psi_{j}^{*}(y)\big| \leq  C_9\, |k_j|^3\,\mu_2^{m_j}\leq  C_9\,a_{m_j+1}^3\,\mu_2^{m_j}$$for any $j\in\{1,...,L\}$ and $y\in\Lambda_{m_j}^f\cap\Lambda_{m_j}^g$, as claimed.

Finally, we deduce from Yoccoz's Lemma \ref{lemyoccoz} (combined with Koebe's distortion principle) that there exists a constant $K>1$, depending only on the real bounds for $f$, such that each $\varphi_j$ is $C^1$ uniformly bounded on $J_{m_j+1}^f\cup\widehat{\Delta}_{m_j}^f$ (or just on $J_{m_j+1}^f$, depending on whether $m_j$ is or is not a two-bridges level).

With these three facts at hand, we estimate the distance between $\varphi_{L}^{*}(x^{*})$ and $\psi_{L}^{*}(y^{*})$ as follows:
\begin{align*}
\big|\varphi_{L}^{*}(x^{*})- \psi_{L}^{*}(y^{*})\big|&\leq
\big|\varphi_{L}^{*}(x^{*})- \varphi_{L}^{*}(y^{*})\big|+\big|\varphi_{L}^{*}(y^{*})- \psi_{L}^{*}(y^{*})\big|\leq K\,|x^{*}-y^{*}|+C_9\,a_{m_L+1}^3\,\mu_2^{m_L}\\
&\leq K\,C_8\,\mu_2^{m}+C_9\,a_{m_L+1}^3\,\mu_2^{m_L}\,.
\end{align*}In the same way, we estimate now the distance between $\varphi_{L-1}^{*}\big(\varphi_{L}^{*}(x^{*})\big)$ and $\psi_{L-1}^{*}\big(\psi_{L}^{*}(y^{*})\big)$\,:
\begin{align*}
\big|\varphi_{L-1}^{*}\big(\varphi_{L}^{*}(x^{*})\big)-\psi_{L-1}^{*}\big(\psi_{L}^{*}(y^{*})\big)\big|&\leq K\,\big|\varphi_{L}^{*}(x^{*})- \psi_{L}^{*}(y^{*})\big|+\big|\varphi_{L-1}^{*}\big(\psi_{L}^{*}(y^{*})\big)-\psi_{L-1}^{*}\big(\psi_{L}^{*}(y^{*})\big)\big|\\
&\leq K^2\,C_8\,\mu_2^{m}+K\,C_9\,a_{m_L+1}^3\,\mu_2^{m_L}+C_9\,a_{m_{L-1}+1}^3\,\mu_2^{m_{L-1}}\,.
\end{align*}
Proceeding inductively, we obtain
\begin{align*}
|v^*-w^*|&=\big|\varphi_1^{*}\circ\cdots\circ\varphi_L^{*}(x^{*})-\psi_1^{*}\circ\cdots\circ\psi_L^{*}(y^{*})\big|\\
&\leq C_{10}\left(K^{L}\,\mu_2^{m}+\sum_{j=1}^{j=L}K^{j-1}\,a_{m_{j}+1}^3\,\mu_2^{m_{j}}\right)\leq C_{10}\,K^{p}\left(\mu_2^{m}+\sum_{j=1}^{j=L}a_{m_{j}+1}^3\,\mu_2^{m_{j}}\right).
\end{align*}By condition \eqref{cond2A} in Definition \ref{setA}, we have $\lim_{j \to \infty} (a_j^3)^{1/j}=1$. Therefore, defining $\varepsilon>0$ by $(1+\varepsilon)\sqrt{\mu_2}=1$, there exists a constant $C_{11}=C_{11}(\varepsilon)>0$ such that $a_j^3 < C_{11}\,(1+ \varepsilon)^j$ for all $j \in \nt$. Hence$$a_{j+1}^3\,\mu_2^{j}<C_{11}\,(1+\varepsilon)^{j+1}\,\mu_2^{j}=C_{11}\,\frac{1}{\sqrt{\mu_2}}\left(\frac{\mu_2}{\sqrt{\mu_2}}\right)^j=C_{12}\,(\sqrt{\mu_2})^j.$$Defining $\mu_*\in(0,1)$ as $\mu_*=\sqrt{\mu_2}=\max\{\mu^{1/2d},\sqrt{\mu_1}\}$, we have
\begin{align*}
|v^*-w^*|&\leq C_{10}\,K^{p}\left(\mu_2^{m}+\sum_{j=1}^{j=L}a_{m_{j}+1}^3\,\mu_2^{m_{j}}\right)\leq C_{10}\,K^{p}\left(\mu_*^{m}+C_{12}\,\sum_{j=1}^{j=L}\mu_*^{m_{j}}\right)\\
&\leq C_{13}\,K^{p}\left(\mu_*^{m}+\sum_{j=n}^{+\infty}\mu_*^{j}\right)=C_{13}\,K^{p}\left(\mu_*^{m}+\frac{\mu_*^n}{1-\mu_*}\right)\leq C_{14}\,K^p\,\mu_*^n\,.
\end{align*}Combining this with hypothesis \eqref{hypTeoB}, we finally obtain
\begin{align*}
|v-w|&=\big|B_{n,f}^{-1}(v^{*})-B_{n,g}^{-1}(w^{*})\big|\leq |I_{n}^f|\,|v^{*}-w^{*}|+\big|\,|I_{n}^f|-|I_{n}^g|\,\big|\,|w^*|\\
&\leq |I_{n}^f|\,C_{14}\,K^p\,\mu_*^n+|I_{n}^f|\,\left|1-\frac{|I_{n}^g|}{|I_{n}^f|}\right|\,\max\{1,s_{n}^g\}\\
&\leq |I_{n}^f|\big(C_{14}\,K^p\,\mu_*^n+C\max\{1,s_{n}^g\}\,\mu^n\big)\leq C_{15}\,K^{p}\,|J_n^{f}|\,\mu_{*}^{n}\,.
\end{align*}This finishes the proof of Lemma \ref{keylemma1}.
\end{proof}

We consider now the fine grid constructed in Section \ref{sec:anewpartition}, in order to establish the final estimates needed for the proof of Theorem \ref{theoremb}. Following \cite[page 367]{dFdM}, we define the \emph{level} of an interval $I \in \mathcal{Q}_n$, denoted $\ell(I)$, as the largest $m \leq n$ such that $I$ is contained in an element of $\widehat{\mathcal{P}}_m$. Theorem \ref{theoremb} will be a straightforward consequence of Proposition \ref{criterion} and the following two lemmas, which are \cite[Lemma 4.12]{dFdM} and \cite[Lemma 4.13]{dFdM} respectively.

\begin{lemma}\label{ineqlevelnm} If $\mathcal{Q}_n$ contains an interval of level $m$, then$$n\leq C_0 \, \sum_{j=1}^{m} \log (1+a_j)$$for some constant $C_0>0$. In particular, if $\rho(f)$ satisfies condition \eqref{cond1A} in Definition \ref{setA}, then $m\geq c_1\, n$, for some constant $0<c_1 <1$ that depends only on $\rho(f)$.
\end{lemma}

We omit the proof of Lemma \ref{ineqlevelnm}, being the same as in \cite[Lemma 4.12]{dFdM}.

\begin{lemma}\label{ineqbeta} If $\rho(f)$ satisfies conditions \eqref{cond2A} and \eqref{cond3A} in Definition \ref{setA}, there exists $\beta\in(0,1)$ with the following property. If $L$ and $R$ are adjacent intervals of $\mathcal{Q}_n$ with $\ell(L)\geq m$ and $\ell(R)\geq m$, then$$\left|\,\frac{|L|}{|R|}-\frac{|h(L)|}{|h(R)|}\,\right| \leq C \,\beta^{m}\,,$$where the constant $C>0$ only depends on the real bounds.
\end{lemma}

\begin{proof}[Proof of Lemma \ref{ineqbeta}] Let us write $m=k+p$ with $p=\lfloor \sigma k \rfloor$, where $\sigma>0$ is a small constant (to be determined along the proof). Let us assume that $L \cup R$ is contained in a single atom $\Delta$ of $\widehat{\mathcal{P}}_{k}$. There are three cases to consider.
\begin{enumerate}
\item [1)] $L \cup R \subset J_k^f(c_0)$. Let $v_1,v_2,v_3$ be the endpoints of $L$ and $R$ respectively, and let $w_1,w_2,w_3$ be the endpoints of $h(L)$ and $h(R)$ respectively. By the triangle inequality
\begin{align}\label{item1lemma5.9}
\left|\,\frac{|L|}{|R|}-\frac{|h(L)|}{|h(R)|}\,\right|
\notag&=\left|\,\frac{|v_1-v_2|}{|v_2-v_3|}-\frac{|w_1-w_2|}{|w_2-w_3|}\,\right|\\
\notag&\leq\left|\,\frac{|v_1-v_2|}{|v_2-v_3|}-\frac{|w_1-w_2|}{|v_2-v_3|}\,\right|+\left|\,\frac{|w_1-w_2|}{|v_2-v_3|}-\frac{|w_1-w_2|}{|w_2-w_3|}\,\right|\\
\notag&\leq\frac{|v_1-w_1|+|v_2-w_2|}{|v_2-v_3|} + \frac{|w_1-w_2|}{|w_2-w_3|}\,\,\frac{|w_2-v_2|+|w_3-v_3|}{|v_2-v_3|}\\
&=\frac{|v_1-w_1|+|v_2-w_2|}{|R|} + \frac{\big|h(L)\big|}{\big|h(R)\big|}\,\,\frac{|w_2-v_2|+|w_3-v_3|}{|R|}\,.
\end{align}We claim that $|h(L)|/|h(R)|$ is (universally) bounded away from zero and infinity. Indeed, note first that, again by the triangle inequality$$\frac{|L|-\sum_{i=1}^{i=2}|v_i-w_i|}{|R|+\sum_{i=2}^{i=3}|v_i-w_i|}\leq\frac{\big|h(L)\big|}{\big|h(R)\big|}\leq\frac{|L|+\sum_{i=1}^{i=2}|v_i-w_i|}{|R|-\sum_{i=2}^{i=3}|v_i-w_i|}\,.$$By Proposition \ref{corokey} we have$$\frac{|L|}{|R|}\,\frac{1-2\,C_1\,\frac{|J_{k}^{f}(c_0)|}{|L|}\,\mu_1^k}{1+2\,C_1\,\frac{|J_{k}^{f}(c_0)|}{|R|}\,\mu_1^k}\leq\frac{\big|h(L)\big|}{\big|h(R)\big|}\leq\frac{|L|}{|R|}\,\frac{1+2\,C_1\,\frac{|J_{k}^{f}(c_0)|}{|L|}\,\mu_1^k}{1-2\,C_1\,\frac{|J_{k}^{f}(c_0)|}{|R|}\,\mu_1^k}\,.$$From Lemma \ref{corolariolemayoccoz} and condition \eqref{cond3A} in Definition \ref{setA} we obtain$$\frac{\big|J_k^f(c_0)\big|}{|L|} \leq M^p\,(a_{k+1}\,a_{k+2}\, \dots \, a_{k+p})^2 \leq M^p\,\exp\big(2 p \, \omega (p/k)\big)\,,$$and the same estimate replacing $L$ with $R$. Let $\beta_1=\big(e^{2\sigma\omega(\sigma)}\,M^{\sigma}\,\mu_1\big)^{1/(2+\sigma)}$,\, and note that $\beta_1\in(0,1)$ by taking $\sigma>0$ small enough. Then we have$$\frac{|L|}{|R|}\,\frac{1-2\,C_1\,\beta_1^{m}}{1+2\,C_1\,\beta_1^{m}}\leq\frac{\big|h(L)\big|}{\big|h(R)\big|}\leq\frac{|L|}{|R|}\,\frac{1+2\,C_1\,\beta_1^{m}}{1-2\,C_1\,\beta_1^{m}}\,.$$By Property \eqref{itemboundgeom} of the auxiliary partitions (Section \ref{secinter}), we finally deduce that the ratio $|h(L)|/|h(R)|$ is bounded, as claimed. With this at hand, and using again Proposition \ref{corokey}, we deduce from \eqref{item1lemma5.9} that$$\left|\,\frac{|L|}{|R|}-\frac{|h(L)|}{|h(R)|}\,\right| \leq \frac{C_2\,\big|J_k^f(c_0)\big|\,\mu_{1}^{k}}{|R|}\,,$$and then$$\left|\,\frac{|L|}{|R|}-\frac{|h(L)|}{|h(R)|}\,\right| \leq C_2\,\beta_1^m\,.$$
\item [2)] $L \cup R \subset J_k^f(c_1)$. As explained right after its proof (see Remark \ref{remkeyprop}), Proposition \ref{corokey} holds in $J_n^{f}(c_1)$. Then we proceed just as in the previous case, using also Lemma \ref{corolariolemayoccoz} (note here that $|\Delta|\asymp|J_k^f(c_1)|$, as it follows from Corollary \ref{corointersect}) and Property \eqref{itemboundgeom} of the auxiliary partitions in the same way.
\item [3)] $L \cup R$ is not contained in $J_k^f(c_0) \cup J_k^f(c_1)$. Let $\Delta^{*}$ be the union of $\Delta$ with its left and right neighbours in the auxiliary partition $\widehat{\mathcal{P}}_{k}$. Let $j<q_{k+1}$ be such that $f^{j}|_{\Delta^*}$ is a diffeomorphism with $f^{j}(\Delta)\subset J_k^f(c_i)$, either for $i=0$ or $i=1$. By the previous cases
\begin{equation}\label{ineqitem1}
\left|\,\frac{|f^j(L)|}{|f^j(R)|}-\frac{|h(f^j(L))|}{|h(f^j(R))|}\,\right| \leq C_2\,\beta_1^m\,.
\end{equation}By Koebe's principle combined with Corollary \ref{mu3}, we deduce in the standard way that the distortion of $f^{j}|_{L \cup R}$ is bounded by $e^{C_3\,\mu^p}$. Therefore, defining $\mu_3=\mu^{\sigma/(2+\sigma)}\in(0,1)$, we obtain
\begin{equation}\label{ineq1lemma3}
\left|\,\frac{\big|f^{j}(L)\big|}{\big|f^{j}(R)\big|} - \frac{|L|}{|R|}\,\right| \leq C_4\,\mu^{p} \leq C_4\,\mu_3^{m}\,.
\end{equation}In the same way, but replacing $f$ by $g$, we obtain
\begin{equation}\label{ineq2lemma3}
\left|\,\frac{\big|g^{j}(h(L))\big|}{\big|g^{j}(h(R))\big|}-\frac{|h(L)|}{|h(R)|}\,\right| \leq C_5\,\mu_3^{m}\,.
\end{equation}By putting together \eqref{ineqitem1}, \eqref{ineq1lemma3} and \eqref{ineq2lemma3}, we finally obtain
\begin{equation*}
\left|\,\frac{\big|h(L)\big|}{\big|h(R)\big|}-\frac{|L|}{|R|}\,\right| \leq C_6\,\beta_2^{m},
\end{equation*}where $C_6=C_2+C_4+C_5$ and $\beta_2= \max\{\mu_{3},\beta_1\}$.
\end{enumerate}
\end{proof}

\subsection*{Proof of Theorem B} Let $\{\mathcal{Q}_n\}_{n\in\nt}$ be the fine grid constructed in Section \ref{sec:anewpartition}, and let $h$ be the topological conjugacy considered in Theorem \ref{theoremb}. Let $c_1\in(0,1)$ be given by Lemma \ref{ineqlevelnm}, and let $C>0$ and $\beta\in(0,1)$ be given by Lemma \ref{ineqbeta}. Then we just apply Proposition \ref{criterion}, with $C$ and $\lambda=\beta^{c_1}$, to deduce that the conjugacy $h$ is a $C^{1+\alpha}$ diffeomorphism.

\medskip
	
\begin{remark}\label{renbimulti} As mentioned in the introduction, the statement of Theorem \ref{maintheorem} is most likely true for multicritical circle maps with any number of critical points, and we believe that it should be possible to adapt the proof of Theorem \ref{theoremb}, developed in Sections \ref{sec:anewpartition} and \ref{SecProofThmB} above, to the multicritical case. To be more precise, let $f$ be a $C^3$ circle homeomorphism with $N \geq 2$ critical points $\{c_0,c_1,\dots,c_{N-1}\}$ (all of them being non-flat) and with irrational rotation number $\rho\in\mathcal{A}$ (recall Definition \ref{setA}). For any given $n\in\nt$, let $N_n\in\{0,\dots,N-1\}$ be the number of critical points of $f^{q_{n+1}}$ that belong to $I_{n}(c_0)\setminus I_{n+2}(c_0)$. Consider the (ordered) set
\[
\big\{\,1\,\leq\,j_1\,<\,\dots\,<\,j_{N_n}\,\leq\,a_{n+1}\,\big\},
\]
where for each $i\in\{1,\dots,N_n\}$ the index $j_i$ is defined by the condition that the fundamental domain$$\Delta_{j_i}=f^{(j_i-1)q_{n+1}+q_n}\big(I_{n+1}(c_0)\big)$$contains a critical point, say $\mathfrak{c}_{n}(i)$, of $f^{q_{n+1}}$. To build auxiliary partitions (recall Section \ref{secinter}), let
\[
\widehat{\Delta}_n(0)=\Delta_1=f^{q_n}\big(I_{n+1}(c_0)\big),\quad\widehat{\Delta}_n(N_n+1)=\Delta_{a_{n+1}}=\big[f^{q_{n+2}}(c_0),f^{(a_{n+1}-1)q_{n+1}+q_n}(c_0)\big]
\]
and consider for each $i\in\{1,\dots,N_n\}$ the fundamental domain
\[
\widehat{\Delta}_n(i)=\big[f^{q_{n+1}}(\mathfrak{c}_{n}(i)),\mathfrak{c}_{n}(i)\big]\,.
\]
Just as we did in Section \ref{secbuildaux}, spread each of these intervals under $f^{q_{n+1}}$, both forward and backwards, until it meets the corresponding iterates of the next and the previous one, $\widehat{\Delta}_n(i+1)$ and $\widehat{\Delta}_n(i-1)$, respectively. With this at hand, it should be possible to build balanced decompositions of bridges and fine grids from the auxiliary partitions, adapting the construction developed in Section \ref{sec:constructingfinegrid}. Moreover, it is reasonable to expect that all geometric estimates of both sections \ref{sec:anewpartition} and \ref{SecProofThmB}, that rely heavily on the real bounds (Theorem \ref{teobeau}), Koebe's distortion principle (see Remark \ref{obspontekoebe}) and Yoccoz's lemma \ref{lemyoccoz}, hold in the same way as for the bicritical case. Note, finally, that the criterion for smoothness given by Proposition \ref{criterion} is quite general, thus independent of the number of critical points of the circle maps referred in the statement of Theorem \ref{theoremb}.
\end{remark}


\begin{thebibliography}{999}
\bibliographystyle{plain}

\bibitem{Av} 
\newblock A.~Avila,
\newblock On rigidity of critical circle maps, 
\newblock \emph{Bull. Braz. Math. Soc.} {\bf 44} (2013), 611--619.

\bibitem{EdF} 
\newblock G.~Estevez and E.~de Faria, 
\newblock Real bounds and quasisymmetric rigidity of multicritical circle maps,
\newblock \emph{Trans. Amer. Math. Soc.} {\bf 370} (2018), 5583--5616.
   
\bibitem{EdFG} 
\newblock G.~Estevez, E.~de Faria and P.~Guarino, 
\newblock Beau bounds for multicritical circle maps,
\newblock \emph{Indag. Math.} {\bf 29} (2018), 842--859.
 
\bibitem{ESY} 
\newblock G.~Estevez, D.~Smania and M.~Yampolsky, 
\newblock Complex bounds for multicritical circle maps with bounded type rotation number,
\newblock available at {\tt{arXiv:2005.02377\/}}. 
 
\bibitem{dF}
\newblock E.~de Faria,
\newblock Asymptotic rigidity of scaling ratios for critical circle mappings,
\newblock \emph{Erg. Theory \& Dyn. Syst.} {\bf 19} (1999), 995--1035.

\bibitem{dFG16} 
\newblock E.~de Faria and P.~Guarino,
\newblock Real bounds and Lyapunov exponents,
\newblock \emph{Disc. and Cont. Dyn. Sys. A} {\bf 36} (2016), 1957--1982.

\bibitem{dFG19} 
\newblock E.~de Faria and P.~Guarino,
\newblock Quasisymmetric orbit-flexibility of multicritical circle maps,
\newblock to appear in \emph{Ergodic Theory and Dynamical Systems}.

\bibitem{dFG20} 
\newblock E.~de Faria and P.~Guarino,
\newblock There are no $\sigma$-finite absolutely continuous invariant measures for multicritical circle maps,
\newblock \emph{Nonlinearity} {\bf 34} (2021), 6727--6749.

\bibitem{dFGsurvey}
\newblock E.~de Faria and P.~Guarino,
\newblock Dynamics of multicritical circle maps,
\newblock to appear in \emph{S\~ao Paulo Journal of Mathematical Sciences (SPJM)}.

\bibitem{dFdM} 
\newblock E.~de Faria and W.~de Melo, 
\newblock Rigidity of critical circle mappings I,
\newblock \emph{J. Eur. Math. Soc.} {\bf 1} (1999), 339--392.

\bibitem{dFdM2} 
\newblock E.~de Faria and W.~de Melo,  
\newblock Rigidity of critical circle mappings II,
\newblock \emph{J. Amer. Math. Soc.} {\bf 13} (2000), 343--370.

\bibitem{GY}
\newblock I.~Gorbovickis and M.~Yampolsky,
\newblock Rigidity, universality, and hyperbolicity of renormalization for critical circle maps with non-integer exponents,
\newblock \emph{Erg. Theory \& Dyn. Syst.} {\bf 40} (2020), 1282--1334.

\bibitem{GdM} 
\newblock P.~Guarino and W.~de Melo, 
\newblock Rigidity of smooth critical circle maps,
\newblock \emph{J. Eur. Math. Soc.} {\bf 19} (2017), 1729--1783.

\bibitem{GMdM} 
\newblock P.~Guarino, M.~Martens and W.~de Melo, 
\newblock Rigidity of critical circle maps,
\newblock \emph{Duke Math. J.} {\bf 167} (2018), 2125--2188.

\bibitem{H} 
\newblock M.~Herman, 
\newblock Conjugaison quasi-sim\'etrique des hom\'eomorphismes du cercle \`a des rotations (manuscript), (1988).
 
\bibitem{khin} 
\newblock A.~Khinchin, 
\newblock Continued fractions,
\newblock Dover Publications Inc. (reprint of the 1964 translation), (1997).

\bibitem{KT}
\newblock K.~Khanin and A.~Teplinsky,
\newblock Robust rigidity for circle diffeomorphisms with singularities,
\newblock \emph{Invent. Math.} {\bf 169} (2007), 193--218.

\bibitem{KY}
\newblock D.~Khmelev and M.~Yampolsky,
\newblock The rigidity problem for analytic critical circle maps,
\newblock \emph{Mosc. Mat. J.} {\bf 6} (2006), 317--351.

\bibitem{dMvS} 
\newblock W.~de Melo and S.~van Strien, 
\newblock One-dimensional dynamics,
\newblock \emph{Springer-Verlag} (1993).

\bibitem{G}
\newblock G.~Światek,
\newblock Rational rotation numbers for maps of the circle,
\newblock \emph{Comm. Math. Phys.} {\bf 119} (1988), 109--128.

\bibitem{Yam} 
\newblock M.~Yampolsky,
\newblock Complex bounds for renormalization of critical circle maps,
\newblock \emph{Ergod. Th. \& Dynam. Sys.} {\bf 19} (1999), 227--257.

\bibitem{Yam1} 
\newblock M.~Yampolsky,
\newblock The attractor of renormalization and rigidity of towers of critical circle maps,
\newblock \emph{Comm. Math. Phys.} {\bf 218} (2001), 537--568.

\bibitem{Yam2} 
\newblock M.~Yampolsky,
\newblock Hyperbolicity of renormalization of critical circle maps,
\newblock \emph{Publ. Math. IHES.} {\bf 96} (2002), 1--41.

\bibitem{Yam3} 
\newblock M.~Yampolsky,
\newblock Renormalization horseshoe for critical circle maps,
\newblock \emph{Comm. Math. Phys.} {\bf 240} (2003), 75--96.

\bibitem{Yam2019} 
\newblock M.~Yampolsky,
\newblock Renormalization of bi-cubic circle maps,
\newblock \emph{C. R. Math. Rep. Acad. Sci. Canada} {\bf 41} (2019), 57--83.

\bibitem{yoccoz} 
\newblock J.-C.~Yoccoz, 
\newblock Il n'y a pas de contre-exemple de Denjoy analytique,
\newblock \emph{C.R. Acad. Sc. Paris} {\bf 298} (1984), 141--144.

\bibitem{zak} 
\newblock S.~Zakeri, 
\newblock Dynamics of cubic Siegel polynomials,
\newblock \emph{Comm. Math. Phys.} {\bf 206} (1999), 185--233.

\end{thebibliography}
\end{document}